\documentclass[12pt,a4paper]{amsart}
\usepackage{amsmath,amsthm,amssymb}
\usepackage{color}
\usepackage[top=30truemm,bottom=30truemm,left=25truemm,right=25truemm]{geometry}
\usepackage{tikz}
\usepackage{hyperref}

\newtheorem{thm}{Theorem}[subsection]
\newtheorem{lem}[thm]{Lemma}
\newtheorem{prop}[thm]{Proposition}
\newtheorem{cor}[thm]{Corollary}

\theoremstyle{definition}
\newtheorem{defi}[thm]{Definition}
\newtheorem{rem}[thm]{Remark}

\newtheorem{ex}[thm]{Example}

\numberwithin{equation}{subsection}

\author{Hideya Watanabe}
\address{(H. Watanabe) Department of Mathematics, Colledge of Science, Rikkyo University, 3-34-1, Nishi-Ikebukuro, Toshima-ku, Tokyo, 171-8501, Japan}
\email{watanabehideya@gmail.com}
\date{\today}

\title[Finite-dimensional irreducible representations of TLA of the 2nd kind]{Finite-dimensional irreducible representations of twisted loop algebras of the second kind}

\begin{document}
\maketitle

\begin{abstract}
  Twisted loop algebras of the second kind are infinite-dimensional Lie algebras that are constructed from a semisimple Lie algebra and an automorphism on it of order at most $2$.
  They are examples of equivariant map algebras.
  The finite-dimensional irreducible representations of an arbitrary equivariant map algebra have been classified by Neher--Savage--Senesi.
  In this paper, we classify the finite-dimensional irreducible representations of twisted loop algebras of the second kind in a more elementary way.
\end{abstract}


\section{Introduction}
\subsection{Twisted loop algebras}
To a simple Lie algebra $\mathfrak{g}$ over $\mathbb{C}$ and a Lie algebra automorphism $\theta$ on $\mathfrak{g}$ of finite order, say $m$, one can associate an infinite-dimensional Lie algebra $\mathcal{L}_1(\mathfrak{g},\theta)$, called the \emph{twisted loop algebra} (of the first kind) (\cite[Chapter 8]{Kac83}).
It is the fixed point subalgebra of the loop algebra $\mathcal{L}(\mathfrak{g}) := \mathbb{C}[t,t^{-1}] \otimes \mathfrak{g}$ with respect to the automorphism $\theta_1$ defined by
\[
  \theta_1(t^n \otimes x) := (\zeta t)^n \otimes \theta(x) \quad \text{ for each } n \in \mathbb{Z},\ x \in \mathfrak{g},
\]
where $\zeta \in \mathbb{C}$ denotes a primitive $m$-th root of $1$.

The construction above can be applied to the current algebra $\mathcal{C}(\mathfrak{g}) := \mathbb{C}[t] \otimes \mathfrak{g}$ instead of the loop algebra.
Let $\mathcal{C}(\mathfrak{g},\theta)$ denote the resulting Lie algebra.

The twisted loop algebra $\mathcal{L}_1(\mathfrak{g},\theta)$ admits a quantum deformation $U_q(\mathcal{L}_1(\mathfrak{g},\theta))$ (\cite[Section 12.2]{ChPr94} for untwisted types, \cite{ChPr98} for twisted types).
Similarly, $\mathcal{C}(\mathfrak{g},\theta)$ has a quantum analogue $Y(\mathfrak{g},\theta)$, called the Yangian (\cite{Dri85} for $\theta = \mathrm{id}$, \cite{GuMa12} for $\theta$ arising from a Dynkin diagram automorphism).
The quantum loop algebra $U_q(\mathcal{L}_1(\mathfrak{g},\theta))$ degenerates to the Yangian $Y(\mathfrak{g,\theta})$ (\cite{Dri87}, \cite{GaTL13}, \cite{GuMa12}).
This explains the similarity of representation theories of these two algebras.

When $m \leq 2$, or equivalently $\theta^2 = \mathrm{id}$, the automorphism $\theta$ can also be extended to $\mathcal{L}(\mathfrak{g})$ by
\[
  \theta_2(t^n \otimes x) := t^{-n} \otimes \theta(x) \quad \text{ for each } n \in \mathbb{Z},\ x \in \mathfrak{g}.
\]
The fixed point subalgebra of $\mathcal{L}(\mathfrak{g})$ with respect to this automorphism is denoted by $\mathcal{L}_2(\mathfrak{g},\theta)$, and called the \emph{twisted loop algebra of the second kind}.

In general, an automorphism $\theta$ on a symmetrizable Kac-Moody Lie algebra $\mathfrak{g}$ is said to be of the first kind if $\theta(\mathfrak{b}^+) \cap \mathfrak{b}^-$ is finite-dimensional, where $\mathfrak{b}^+$ and $\mathfrak{b}^-$ are the Borel subalgebra and the opposite Borel subalgebra of $\mathfrak{g}$, respectively, with respect to a fixed simple system for the root system of $\mathfrak{g}$.
Similarly, $\theta$ is said to be of the second kind if $\theta(\mathfrak{b}^+) \cap \mathfrak{b}^+$ is finite-dimensional.
Clearly, when $\mathfrak{g}$ is finite-dimensional, each automorphism is of the first and the second kind at the same time.
When $\theta$ arises from a Dynkin diagram automorphism, the automorphism $\theta_1$ above is the restriction of an automorphism of the first kind on the affine Lie algebra associated with $(\mathfrak{g},\theta)$, while $\theta_2$ is the restriction of an automorphism of the second kind.

The automorphisms of the second kind of order at most $2$ are main ingredients in Kolb's construction of quantum symmetric pairs in \cite{Kol14}.
In particular, the twisted loop algebra $\mathcal{L}_2(\mathfrak{g},\theta)$ of the second kind can be quantized to a coideal subalgebra $\mathbf{U}^\imath(\mathcal{L}_2(\mathfrak{g},\theta))$ of the quantum loop algebra $U_q(\mathcal{L}_1(\mathfrak{g},\mathrm{id}))$.

Although the construction above cannot be applied to the current algebra, in some cases there are related quantum algebras, called \emph{twisted Yangians} (\cite{Ols92}, \cite{GuRe16}).
Twisted Yangians are quantum deformations of $\mathcal{C}(\mathfrak{g},\theta)$, and coideal subalgebras of Yangians.
It has been proved in \cite{CoGu15} and \cite{LWZ24} that in some cases, $\mathbf{U}^\imath(\mathcal{L}_2(\mathfrak{g},\theta))$ degenerates to a twisted Yangian.

Compared to the quantum loop algebras, less is known about finite-dimensional representation theory of $\mathbf{U}^\imath(\mathcal{L}_2(\mathfrak{g},\theta))$.
There are some type-dependent classification results; \cite{ItTe10} for the $q$-Onsager algebra (type $A\mathrm{I}_1$) and \cite{GoMo10} for type $A\mathrm{II}$.
Finite-dimensional representations have been studied from a viewpoint of integrable systems in \cite{KOY19}, \cite{KuOk22}, \cite{ApVl22b}, \cite{KOW22}, to name a few.
Recently, there is a progress on character theory for type $A\mathrm{I}$ in \cite{LiPr25}.

\subsection{Equivariant map algebras}
At the classical level ($q = 1$), the finite-dimensional irreducible representations of a twisted loop algebra of the second kind has been classified in the theory of \emph{equivariant map algebras}, which was developed in \cite{NSS12} (see also a survey paper \cite{NeSa13}).

The equivariant map algebra associated with a Lie algebra $\mathfrak{g}$, a scheme $X$, and a finite group $\Gamma$ acting on both $\mathfrak{g}$ and $X$, is the Lie algebra of equivariant regular maps $X \to \mathfrak{g}$ of schemes with respect to the $\Gamma$-actions.
The current algebras, loop algebras, and twisted loop algebras of the first and second kinds are all equivariant map algebras.
Although the equivariant map algebras form a huge class of Lie algebras, their finite-dimensional irreducible representations are classified in a unified way.

\subsection{Classification of finite-dimensional irreducible representations}
The aim of this paper is to classify the finite-dimensional irreducible representations of each twisted loop algebra of the second kind in an elementary way.
As mentioned above, the classification problem has been settled in the theory of equivariant map algebras.
However, taking quantum analogues into account, a proof of the classification theorem should be as elementary as possible.
Also, from the viewpoint of semisimple Lie algebras and loop algebras, it is more natural to classify the finite-dimensional irreducible representations in terms of highest weights.

The classification of finite-dimensional irreducible representations of a semisimple Lie algebra is reduced to that of the minimal semisimple Lie algebra; $\mathfrak{sl}_2$.
Let $e,f,h \in \mathfrak{sl}_2$ denote the Chevalley basis.
Roughly speaking, the classification for $\mathfrak{sl}_2$ is achieved by calculating the products $e f^r$ for all $r \in \mathbb{Z}_{\geq 0}$.

The situation is very similar for the loop algebras.
In this case, the smallest loop algebra is $\mathcal{L}(\mathfrak{sl}_2)$, and we need to compute $e_1^r f_0^r$, where $x_a := t^a \otimes x$ for each $a \in \mathbb{Z}$ and $x \in \mathfrak{sl}_2$.

The strategy in this paper is an analogue of the above.
However, unlike semisimple Lie algebras and loop algebras, there are four minimal twisted loop algebras of the second kind.
They are the loop algebra of $\mathfrak{sl}_2$, the current algebra of $\mathfrak{sl}_2$, the Onsager algebra (\cite{Ons44}, see also \cite{Roa91} for relations with $\mathcal{L}(\mathfrak{sl}_2)$), and the generalized Onsager algebra of $\mathfrak{sl}_3$ (\cite{UgIv96}, see also \cite{Sto20} for general types).
As a result, we need to compute something like $e_1^r f_0^r$ for these algebras one by one.
Theorems \ref{thm: classification DeltaA1}, \ref{thm: classification A1}, \ref{thm: classification AI1}, and \ref{thm: classification AI2} classify the finite-dimensional irreducible representations of these four algebras in terms of highest weights.
These results lead us to the classification theorem for a general twisted loop algebra of the second kind (Theorem \ref{thm: classification general}).

\subsection{Organization}
This paper is organized as follows.
In Section \ref{sect: tla2}, we introduce the twisted loop algebras of the second kind, and explain how to describe them in terms of Dynkin diagrams.
In Section \ref{sect: ema}, we review the theory of equivariant map algebras and then apply it to the twisted loop algebras of the second kind.
After fixing notation and proving some preliminary results in Section \ref{sect: prelim}, we classify the finite-dimensional irreducible representations of each twisted loop algebra of the second kind in terms of highest weights in Section \ref{sect: classification}.

\subsection*{Acknowledgments}
The author thanks Hironori Oya for explaining him about the representation theory of quantum affine $\mathfrak{sl}_2$.
He is also grateful to Tomasz Prze\'{z}dziecki for discussion on the representation theory of affine quantum symmetric pairs.
This work was supported by JSPS KAKENHI Grant Number JP24K16903.

\section{Twisted loop algebras of the second kind}\label{sect: tla2}
In this section, we introduce the twisted loop algebras of the second kind.
Then, we explain that each twisted loop algebra of the second kind can be described in terms of Dynkin diagrams.
This enables us to formulate the notion of highest weight modules for the twisted loop algebras of the second kind.

For the terminology regarding semisimple and affine Lie algebras, we refer \cite{Kac83}.

\subsection{Basic definitions and properties}
For each Lie algebra $\mathfrak{a}$, let $\operatorname{Aut}(\mathfrak{a})$ denote the group of automorphisms on $\mathfrak{a}$.

Let $\mathfrak{g}$ be a semisimple Lie algebra over $\mathbb{C}$ and $\theta \in \operatorname{Aut}(\mathfrak{g})$ such that $\theta^2 = \mathrm{id}$.
Set
\[
  \mathfrak{k} := \{ x \in \mathfrak{g} \mid \theta(x) = x \}, \quad \mathfrak{p} := \{ x \in \mathfrak{g} \mid \theta(x) = -x \}.
\]
Then, we have a $\mathbb{Z}/2\mathbb{Z}$-gradation
\begin{align}\label{eq: Z2_grad_g}
  \mathfrak{g} = \mathfrak{k} \oplus \mathfrak{p}.
\end{align}
Namely, we have $[\mathfrak{k}, \mathfrak{k}], [\mathfrak{p},\mathfrak{p}] \subseteq \mathfrak{k}$ and $[\mathfrak{k},\mathfrak{p}] \subseteq \mathfrak{p}$.

Let $L := \mathbb{C}[t,t^{-1}]$ denote the $\mathbb{C}$-algebra of Laurent polynomials in one variable $t$, and $\mathcal{L} := L \otimes_\mathbb{C} \mathfrak{g}$ the loop algebra associated with $\mathfrak{g}$.
Namely, $\mathcal{L}$ is the Lie algebra over $\mathbb{C}$ with Lie bracket given by
\begin{align*}\label{eq: Lie bracket for loop alg}
  [P \otimes x, Q \otimes y] := PQ \otimes [x,y] \ \text{ for each } P,Q \in L,\ x,y \in \mathfrak{g}.
\end{align*}
The automorphism $\theta$ on $\mathfrak{g}$ can be extended to $\mathcal{L}$ by
\[
  \theta(t^n \otimes x) := t^{-n} \otimes \theta(x) \ \text{ for each } n \in \mathbb{Z},\ x \in \mathfrak{g}.
\]

\begin{defi}
  The \emph{twisted loop algebra of the second kind} associated with $(\mathfrak{g}, \theta)$ is the fixed point subalgebra $\mathcal{L}^{\mathrm{tw}}(\mathfrak{g},\theta)$ of $\mathcal{L}$ with respect to $\theta$:
  \[
    \mathcal{L}^{\mathrm{tw}}(\mathfrak{g},\theta) := \{ \mathfrak{x} \in \mathcal{L} \mid \theta(\mathfrak{x}) = \mathfrak{x} \}.
  \]
\end{defi}


The Lie algebra $\mathcal{L}^\mathrm{tw} := \mathcal{L}^{\mathrm{tw}}(\mathfrak{g},\theta)$ has a natural $\mathbb{Z}/2\mathbb{Z}$-gradation.
To describe it, set
\begin{align}\label{eq: Lpm}
  L_\pm := \{ P(t) \in L \mid P(t^{-1}) = \pm P(t) \}.
\end{align}

\begin{prop}
  We have the following $\mathbb{Z}/2\mathbb{Z}$-gradation of $\mathcal{L}^\mathrm{tw}$$:$
  \[
    \mathcal{L}^\mathrm{tw} = (L_+ \otimes \mathfrak{k}) \oplus (L_- \otimes \mathfrak{p}).
  \]
\end{prop}

\begin{proof}
  By the decomposition \eqref{eq: Z2_grad_g}, we have
  \[
    \mathcal{L} = (L \otimes \mathfrak{k}) \oplus (L \otimes \mathfrak{p})
  \]
  as linear spaces.
  Since $\theta$ preserves this decomposition, the assertion follows.
\end{proof}

Let us see the effect of conjugating $\theta$ by a Lie algebra automorphism $\phi \in \operatorname{Aut}(\mathfrak{g})$.
The conjugate ${}^\phi \theta := \phi \circ \theta \circ \phi^{-1}$ satisfies $({}^\phi \theta)^2 = \mathrm{id}$.
Hence, we can consider the corresponding twisted loop algebra $\mathcal{L}^{\mathrm{tw}}(\mathfrak{g},{}^\phi \theta)$ of the second kind.

On the other hand, $\phi$ can be extended to $\mathcal{L}$ by
\[
  \phi(t^n \otimes x) := t^n \otimes \phi(x) \ \text{ for each } n \in \mathbb{Z},\ x \in \mathfrak{g}.
\]

\begin{prop}\label{prop: conjugation}
  The automorphism $\phi$ on $\mathcal{L}$ gives rise to a Lie algebra isomorphism
  \[
    \mathcal{L}^{\mathrm{tw}}(\mathfrak{g},\theta) \rightarrow \mathcal{L}^{\mathrm{tw}}(\mathfrak{g},{}^\phi \theta).
  \]
\end{prop}

\begin{proof}
  The assertion follows immediately from the definitions.
\end{proof}

\subsection{Involutions}\label{subsect: involution}
Proposition \ref{prop: conjugation} implies that as far as we are interested in the algebraic structure of a twisted loop algebra $\mathcal{L}^{\mathrm{tw}}(\mathfrak{g},\theta)$ of the second kind, we may freely replace the automorphism $\theta$ by its conjugation.
When $\theta = \mathrm{id}$, there is no choice.
Hence, let us assume that $\theta \neq \mathrm{id}$.
The automorphisms of finite order on each simple Lie algebra is classified in \cite[Section 8.6]{Kac83}.
In our case, it is formulated as follows.

Let $I = \{ 1,\dots,n \}$ be a Dynkin diagram of irreducible finite type, say $X_n$, and $\mathfrak{g} = \mathfrak{g}(I)$ denote the corresponding simple Lie algebra with Chevalley generators $\{ e_i,f_i,h_i \mid i \in I \}$.
Let $\mathfrak{h} \subset \mathfrak{g}$ denote the Cartan subalgebra generated by $\{ h_i \mid i \in I \}$, $\alpha_i \in \mathfrak{h}^*$ the root of $e_i$ for each $i \in I$, and $\Phi^+$ the set of positive roots.
For each $\alpha \in \Phi^+$, let $\mathfrak{g}_{\pm \alpha}$ denote the root space of root $\pm \alpha$.

Let $\operatorname{Aut}(I)$ denote the group of automorphisms on $I$.
Let $\mu \in \operatorname{Aut}(I)$ be such that $\mu^2 = \mathrm{id}$, and $r \in \{ 1,2 \}$ denote the order of $\mu$.
There exists a unique Lie algebra automorphism $\mu$ on $\mathfrak{g}$ such that
\[
  \mu(e_i) = e_{\mu(i)}, \quad \mu(f_i) = f_{\mu(i)}, \quad \mu(h_i) = h_{\mu(i)} \text{ } \text{ for all } i \in I.
\]

Let $\tilde{I} = \{ 0,1,\dots,l \}$ denote the Dynkin diagram of affine type $X_n^{(r)}$.
We slightly change the numbering of vertices from \cite{Kac83} so that $\{ 1,\dots,l \}$ in $I$ form a complete set of representatives for the $\mu$-orbits in $I$.
For example, the vertices $4,5,6$ for the Dynkin diagram of type $E_6$ in \cite{Kac83} are $5,6,4$ here.
Define $\alpha^0 \in \Phi^+$ and $h_0 \in \mathfrak{h}$ by
\[
  \alpha^0 := \begin{cases}
    \sum_{j=1}^l a_j \alpha_j & \text{ if } r = 1, \\
    \alpha_1 + \dots + \alpha_l & \text{ if } X_n^{(r)} = D_{l+1}^{(2)}, \\
    \alpha_1 + \dots + \alpha_{2l-2} & \text{ if } X_n^{(r)} = A_{2l-1}^{(2)}, \\
    \alpha_1 + 2\alpha_2 + 2\alpha_3 + \alpha_4 + \alpha_5 + \alpha_6 & \text{ if } X_n^{(r)} = E_{6}^{(2)}, \\
    \alpha_1 + \dots + \alpha_{2l} & \text{ if } X_n^{(r)} = A_{2l}^{(2)},
  \end{cases}
\]
\[
  h_0 := \begin{cases}
    -\sum_{j=1}^l a_j^\vee h_j & \text{ if } r = 1, \\
    -(2(h_1 + \dots + h_{l-1}) + h_l + h_{l+1}) & \text{ if } X_n^{(r)} = D_{l+1}^{(2)}, \\
    -(h_1 + 2(h_2 + \dots + h_{2l-2}) + h_{2l-1}) & \text{ if } X_n^{(r)} = A_{2l-1}^{(2)}, \\
    -(2h_1 + 3h_2 + 4h_3 +2h_4 +3h_5 + 2h_6) & \text{ if } X_n^{(r)} = E_{6}^{(2)}, \\
    -(h_1 + \dots + h_{2l}) & \text{ if } X_n^{(r)} = A_{2l}^{(2)},
  \end{cases}
\]
where $a_j$ and $a_j^\vee$ are as in \cite[Sections 4.8, 6.1]{Kac83}.

Then, we can choose root vectors $e_{\pm \alpha^0} \in \mathfrak{g}_{\pm \alpha^0}$ satisfying the following:
Set
\begin{align*}
  &e_0 := \begin{cases}
    e_{-\alpha^0} & \text{ if } r = 1 \text{ or } X_n^{(r)} = A_{2l}^{(2)}, \\
    e_{-\alpha^0} - \mu(e_{-\alpha^0}) & \text{ otherwise},
  \end{cases} \\
  &f_0 := \begin{cases}
    -e_{\alpha^0} & \text{ if } r = 1 \text{ or } X_n^{(r)} = A_{2l}^{(2)}, \\
    -e_{\alpha^0} + \mu(e_{\alpha^0}) & \text{ otherwise}.
  \end{cases}
\end{align*}
Then, $(e_0,f_0,h_0)$ forms an $\mathfrak{sl}_2$-triple:
\begin{align}\label{eq: sl2 triple 0}
  [e_0,f_0] = h_0, \quad [h_0,e_0] = 2e_0, \quad [h_0,f_0] = -2f_0.
\end{align}

Let us extend $\mu$ on $I$ to $I \sqcup \{ 0 \}$ by setting $\mu(0) := 0$.
For each $i,j \in I$, let $a_{i,j} \in \mathbb{Z}$ denote the corresponding Cartan integer.
For each $j \in \{ 0,1,\dots,l \} \subseteq I \sqcup \{ 0 \}$, set
\begin{align*}
  &x_j := \begin{cases}
    e_j & \text{ if } \mu(j) = j, \\
    e_j + e_{\mu(j)} & \text{ if } a_{j,\mu(j)} = 0, \\
    \sqrt{2}(e_j + e_{\mu(j)}) & \text{ if } a_{j,\mu(j)} = -1,
  \end{cases} \\
  &y_j := \begin{cases}
    f_j & \text{ if } \mu(j) = j, \\
    f_j + f_{\mu(j)} & \text{ if } a_{j,\mu(j)} = 0, \\
    \sqrt{2}(f_j + f_{\mu(j)}) & \text{ if } a_{j,\mu(j)} = -1,
  \end{cases} \\
  &w_j := \begin{cases}
    h_j & \text{ if } \mu(j) = j, \\
    h_j + h_{\mu(j)} & \text{ if } a_{j,\mu(j)} = 0, \\
    2(h_j + h_{\mu(j)}) & \text{ if } a_{j,\mu(j)} = -1.
  \end{cases}
\end{align*}
Note that if $X_n^{(r)} \neq A_{2l}^{(2)}$, then we have
\begin{align}\label{eq: w0}
  w_0 = h_0 = -\sum_{j=1}^l a_j^\vee w_j
\end{align}

\begin{prop}[{\cite[Section 8.6]{Kac83}}]\label{prop: classification involution}
  \hfill
  \begin{enumerate}
    \item\label{item: classification involution construction} For each $\mathbf{s} = (s_0,s_1,\dots,s_l) \in \{0,1\}^{\tilde{I}}$ such that $r \sum_{j \in \tilde{I}} a_js_j = 2$, there exists a unique involution $\sigma_\mathbf{s}$ on $\mathfrak{g}$ such that
    \[
      \sigma_\mathbf{s}(x_j) = (-1)^{s_j}x_j \ \text{ for all } j \in \tilde{I}.
    \]
    \item\label{item: classification involution conjugation} Two automorphisms of the form $\sigma_\mathbf{s}$ and $\sigma_{\mathbf{s}'}$ are conjugate to each other if and only if there exists $\tau \in \operatorname{Aut}(\tilde{I})$ such that $s_j = s'_{\tau(j)}$ for all $j \in \tilde{I}$.
    \item\label{item: classification involution exhaustion} Each involution on $\mathfrak{g}$ is conjugate to some $\sigma_\mathbf{s}$.
    \item\label{item: classification involution frk} The fixed point subalgebra $\mathfrak{k}$ of $\mathfrak{g}$ with respect to $\sigma_\mathbf{s}$ is of the form $\mathfrak{k} = \mathfrak{k}' \oplus \mathfrak{z}(\mathfrak{k})$, where $\mathfrak{k}'$ is a semisimple subalgebra of $\mathfrak{g}$ and $\mathfrak{z}(\mathfrak{k})$ denotes the center of $\mathfrak{k}$, which is a subspace of $\bigoplus_{j \in \tilde{I}} \mathbb{C} w_j$ of dimension
    \[
      \sharp\{ j \in \tilde{I} \mid s_j = 1 \}-1.
    \]
    Moreover, the Dynkin diagram of $\mathfrak{k}'$ is isomorphic to $\{ j \in \tilde{I} \mid s_j = 0 \}$, and the elements $\{ x_j,y_j,w_j \mid j \in \tilde{I} \text{ such that } s_j = 0 \}$ are the Chevalley generators.
  \end{enumerate}
\end{prop}

\begin{defi}
  Let $\tilde{I}$ be a Dynkin diagram of affine type, say $X_n^{(r)}$ for some $r \in \{ 1,2 \}$, and $\tilde{S} \subseteq \tilde{I}$ a subset such that $r \sum_{j \in \tilde{S}} a_j = 2$.
  Let $\mathbf{s} = (s_0,s_1,\dots,s_l) \in \{ 0,1 \}^{\tilde{I}}$ denote the sequence such that $s_j = 1$ if and only if $j \in \tilde{S}$, and set $\theta(\tilde{S}) := \sigma_\mathbf{s}$.
  The \emph{twisted loop algebra of the second kind} $\mathcal{L}^{\mathrm{tw}}(\tilde{I},\tilde{S})$ associated with $(\tilde{I}, \tilde{S})$ is the twisted loop algebra of the second kind associated with $(\mathfrak{g},\theta(\tilde{S}))$:
  \[
    \mathcal{L}^{\mathrm{tw}}(\tilde{I},\tilde{S}) := \mathcal{L}^{\mathrm{tw}}(\mathfrak{g},\theta(\tilde{S})).
  \]
\end{defi}

By Proposition \ref{prop: classification involution} \eqref{item: classification involution exhaustion}, each involution on $\mathfrak{g}$ is conjugate to some $\theta(\tilde{S})$.
Following Proposition \ref{prop: classification involution} \eqref{item: classification involution conjugation}, one can make a complete list of representatives for the conjugacy classes of involutions in terms of $(\tilde{I}, \tilde{S})$.
The result is in the end of the paper; Figures \ref{fig 1}--\ref{fig 3}, where we painted the vertices in $\tilde{S}$ black.
The diagrams of finite type in Figure \ref{fig 3} will be explained later.

\subsection{Triangular decompositions and highest weight modules}
Let $I$ be a Dynkin diagram of finite type, and $\mathfrak{g} := \mathfrak{g}(I)$ denote the corresponding semisimple Lie algebra over $\mathbb{C}$ with Chevalley generators $\{ e_i,f_i,h_i \mid i \in I \}$.

For each $\mu \in \operatorname{Aut}(I)$, set
\[
  I^\mu := \{ i \in I \mid \mu(i) = i \}.
\]

Let $\mu \in \operatorname{Aut}(I)$ be such that $\mu^2 = \mathrm{id}$, and $S$ a subset of $I^\mu$.
Then, there exists a unique $\theta := \theta(\mu,S) \in \operatorname{Aut}(\mathfrak{g})$ such that
\begin{align}\label{eq: theta(I,sigma,J)}
  \begin{split}
    \theta(e_i) = \begin{cases}
      e_{\mu(i)} & \text{ if } i \notin S,\\
      -e_i & \text{ if } i \in S,
    \end{cases}
    \quad
    \theta(f_i) = \begin{cases}
      f_{\mu(i)} & \text{ if } i \notin S,\\
      -f_i & \text{ if } i \in S,
    \end{cases}
    \quad
    \theta(h_i) = h_{\mu(i)}.
  \end{split}
\end{align}
Clearly, we have $\theta^2 = \mathrm{id}$.

\begin{defi}
  The \emph{twisted loop algebra of the second kind} $\mathcal{L}^{\mathrm{tw}}(I,\mu,S)$ associated with $(I,\mu,S)$ is the twisted loop algebra of the second kind associated with $(\mathfrak{g}(I),\theta(\mu,S))$:
  \[
    \mathcal{L}^{\mathrm{tw}}(I,\mu,S) := \mathcal{L}^{\mathrm{tw}}(\mathfrak{g}(I),\theta(\mu,S)).
  \]
\end{defi}

\begin{rem}
  Let $I = \{ 1,\dots,n \}$ be a Dynkin diagram of irreducible finite type and $\tilde{I} = \{ 0,1,\dots,l \}$ the corresponding Dynkin diagram of affine type.
  For each $(\tilde{I},\tilde{S})$ in Figures \ref{fig 1}--\ref{fig 3}, we have $\theta(\tilde{S}) = \theta(\mu,S)$, where $S := \tilde{S} \setminus \{ 0 \}$.
  For $(\tilde{I},\tilde{S})$ in Figure \ref{fig 1} or \ref{fig 2}, $(I,S)$ is obtained from it by removing the vertex $0$.
  The Dynkin diagrams of finite type in Figure \ref{fig 3} represents $(I,\mu,S)$.
\end{rem}

Let $\mathcal{L} := L \otimes \mathfrak{g}$ denote the loop algebra, and $\mathcal{L}^\mathrm{tw} := \mathcal{L}^{\mathrm{tw}}(I,\mu,S)$ the twisted loop algebra of the second kind.

Let $\Phi$ denote the set of roots, and $\alpha_i \in \Phi$ the root of $e_i$ for each $i \in I$.
Let $\Phi^\pm \subset \Phi$ denote the set of positive and negative roots, respectively.
For each $\alpha \in \Phi$, let $\mathfrak{g}_\alpha$ denote the root space.
Then, we have a triangular decomposition
\begin{align}\label{eq: tri_decomp_g}
  \mathfrak{g} = \mathfrak{n}^- \oplus \mathfrak{h} \oplus \mathfrak{n}^+,
\end{align}
where
\[
  \mathfrak{n}^\pm := \bigoplus_{\alpha \in \Phi^\pm} \mathfrak{g}_\alpha, \quad \mathfrak{h} := \bigoplus_{i \in I} \mathbb{C} h_i.
\]

By the construction \eqref{eq: theta(I,sigma,J)}, the automorphism $\theta := \theta(\mu,S)$ preserves the triangular decomposition.
Hence, we obtain
\[
  \mathfrak{k} = \mathfrak{k}^- \oplus \mathfrak{k}^0 \oplus \mathfrak{k}^+,\quad \mathfrak{p} = \mathfrak{p}^- \oplus \mathfrak{p}^0 \oplus \mathfrak{p}^+,
\]
where
\[
  \mathfrak{a}^\pm := \mathfrak{a} \cap \mathfrak{n}^\pm, \quad \mathfrak{a}^0 := \mathfrak{a} \cap \mathfrak{h} \quad \text{ for each } \mathfrak{a} \in \{ \mathfrak{k},\mathfrak{p} \}.
\]

The loop algebra $\mathcal{L}$ inherits the triangular decomposition \eqref{eq: tri_decomp_g}:
\begin{align}\label{eq: tri decomp clL}
  \mathcal{L} = \mathcal{L}^- \oplus \mathcal{L}^0 \oplus \mathcal{L}^+,
\end{align}
where
\[
  \mathcal{L}^\pm := L \otimes \mathfrak{n}^\pm, \quad \mathcal{L}^0 := L \otimes \mathfrak{h}.
\]

The automorphism $\theta$ on $\mathcal{L}$ also preserves the triangular decomposition.
Hence, we obtain
\begin{align}\label{eq: tri decomp clLtw}
  \mathcal{L}^\mathrm{tw} = \mathcal{L}^{\mathrm{tw}, -} \oplus \mathcal{L}^{\mathrm{tw}, 0} \oplus \mathcal{L}^{\mathrm{tw}, +},
\end{align}
where
\[
  \mathcal{L}^{\mathrm{tw}, \bullet} := \mathcal{L}^\mathrm{tw} \cap \mathcal{L}^\bullet = (L_+ \otimes \mathfrak{k}^\bullet) \oplus (L_- \otimes \mathfrak{p}^\bullet) \ \text{ for each } \bullet \in \{ -,0,+ \};
\]
see equation \eqref{eq: Lpm} for the definitions of $L_\pm$.

Let $U$ denote the universal enveloping algebra of $\mathcal{L}$, and $U^\bullet$ the subalgebra generated by $\mathcal{L}^\bullet$ for each $\bullet \in \{ -,0,+ \}$.
Then, the triangular decomposition \eqref{eq: tri decomp clL} induces
\[
  U = U^- \otimes U^0 \otimes U^+.
\]

Similarly, let $U^\mathrm{tw}$ denote the universal enveloping algebra of $\mathcal{L}^\mathrm{tw}$, and $U^{\mathrm{tw},\bullet}$ the subalgebra generated by $\mathcal{L}^{\mathrm{tw},\bullet}$ for each $\bullet \in \{ -,0,+ \}$.
Then, the triangular decomposition \eqref{eq: tri decomp clLtw} induces
\begin{align}\label{eq: tri decomp Utw}
  U^\mathrm{tw} = U^{\mathrm{tw},-} \otimes U^{\mathrm{tw},0} \otimes U^{\mathrm{tw},+}.
\end{align}

Now, it is clear how to define the notion of highest weight modules over $\mathcal{L}^{\mathrm{tw}}$.

\begin{defi}
  An $\mathcal{L}^\mathrm{tw}$-module $V$ is said to be a \emph{highest weight module} of highest weight $\phi \in (\mathcal{L}^{\mathrm{tw},0})^* := \operatorname{Hom}_\mathbb{C}(\mathcal{L}^{\mathrm{tw},0},\mathbb{C})$ if there exists a nonzero element $v_\phi \in V$, called a \emph{highest weight vector}, such that
  \begin{itemize}
    \item $\mathcal{L}^{\mathrm{tw},+} v_\phi = 0$,
    \item $w v_\phi = \phi(w) v_\phi$ for all $w \in \mathcal{L}^{\mathrm{tw},0}$,
    \item $U^\mathrm{tw} v_\phi = V$.
  \end{itemize}
\end{defi}

Using the triangular decomposition \eqref{eq: tri decomp Utw}, for each $\phi \in (\mathcal{L}^{\mathrm{tw},0})^*$, we can construct a universal highest weight module $M(\phi)$ of highest weight $\phi$:
\[
  M(\phi) := U^\mathrm{tw}/(U^{\mathrm{tw},>0} + \sum_{w \in \mathcal{L}^{\mathrm{tw},0}} U^\mathrm{tw}(w-\phi(w))),
\]
where $U^{\mathrm{tw},>0}$ denotes the left ideal of $U^{\mathrm{tw}}$ generated by $\mathcal{L}^{\mathrm{tw},+}$.
It is a straightforward analogue of the highest weight theory for simple Lie algebras to verify that each highest weight module of highest weight $\phi$ is a quotient of $M(\phi)$, and $M(\phi)$ has a unique simple quotient $V(\phi)$.
Also, each finite-dimensional module is a sum of highest weight modules.
Therefore, the classification of finite-dimensional irreducible representations is reduced to determining for which $\phi \in (\mathcal{L}^{\mathrm{tw},0})^*$ the module $V(\phi)$ is finite-dimensional.

\section{Equivariant Map Algebras}\label{sect: ema}
In this section, we review the general theory of equivariant map algebras developed in \cite{NSS12}.
Then, we explain that each twisted loop algebra of the second kind is an equivariant map algebra, and apply the theory to it.

As \cite{NSS12}, we refer \cite{EiHa00} for the terminology regarding schemes.

\subsection{General theory}
Let $R$ be a unital associative commutative $\mathbb{C}$-algebra, and $\mathfrak{g}$ a finite-dimensional Lie algebra over $\mathbb{C}$.
Set $X := \operatorname{Spec} R$, the affine scheme over $\mathbb{C}$ associated with $R$.
We also regard $\mathfrak{g}$ as an affine space.

\begin{defi}[{\cite[Definition 2.1]{NSS12}}]
  The \emph{map algebra} associated with $(X,\mathfrak{g})$ is the algebra $M(X,\mathfrak{g})$ of regular maps from $X$ to $\mathfrak{g}$.
\end{defi}

Let $\Gamma$ be a finite group acting on both $X$ and $\mathfrak{g}$.

\begin{defi}[{\cite[Definition 3.1]{NSS12}}]
  The \emph{equivariant map algebra} associated with $(X,\mathfrak{g},\Gamma)$ is the algebra $M(X,\mathfrak{g})^\Gamma$ of $\Gamma$-equivariant regular maps from $X$ to $\mathfrak{g}$.
\end{defi}
  
The algebras of regular maps $X \to \mathfrak{g}$ can be identified with $R \otimes \mathfrak{g}$.
The $\Gamma$-action on $X$ induces a $\Gamma$-action on $R$, and hence on $R \otimes \mathfrak{g}$.
We shall identify $M(X,\mathfrak{g})^\Gamma$ with
\[
  \mathfrak{M} := \{ f \in R \otimes \mathfrak{g} \mid \gamma f = f \ \text{ for each } \gamma \in \Gamma \}.
\]

Let $X_\mathrm{rat}$ denote the set of rational points of $X$.
For each $x \in X_\mathrm{rat}$, set
\begin{align*}
  &\Gamma_x := \{ \gamma \in \Gamma \mid \gamma x = x \}, \\
  &\mathfrak{g}^x := \{ g \in \mathfrak{g} \mid \gamma g = g \ \text{ for each } \gamma \in \Gamma_x \}.
\end{align*}
Then, there exists a surjective Lie algebra homomorphism
\[
  \mathrm{ev}_x : \mathfrak{M} \to \mathfrak{g}^x; \ r \otimes g \mapsto r(x) g,
\]
called the \emph{evaluation map} at $x$.
For each $\mathfrak{g}^x$-module $V$, let $V_x$ denote the $\mathfrak{M}$-module with underlying linear space $V$ and action given by
\[
  f v := \mathrm{ev}_x(f) v \ \text{ for each } f \in \mathfrak{M}, \ v \in V.
\]

For each Lie algebra $\mathfrak{a}$, let $\mathfrak{a}_\mathrm{ab}$ denote its abelianization $\mathfrak{a}/[\mathfrak{a},\mathfrak{a}]$.
Each $\chi \in \mathfrak{a}_\mathrm{ab}^* := \operatorname{Hom}_\mathbb{C}(\mathfrak{a}_\mathrm{ab},\mathbb{C})$ defines a one-dimensional representation of $\mathfrak{a}$, which we denote by $\mathbb{C}_\chi$.

\begin{thm}[{\cite[Theorem 5.5]{NSS12}}]\label{thm: classification fd irrep ema}
  Let $V$ be a finite-dimensional simple $\mathfrak{M}$-module.
  Then, there exist $\chi \in \mathfrak{M}_\mathrm{ab}^*$, $n \in \mathbb{Z}_{\geq 0}$, $x_1,\dots,x_n \in X_\mathrm{rat}$ with distinct $\Gamma$-orbits, and a finite-dimensional simple $\mathfrak{g}^{x_i}$-module $V_i$ for each $i = 1,\dots,n$ such that
  \[
    V \simeq \mathbb{C}_\chi \otimes (V_1)_{x_1} \otimes \cdots \otimes (V_n)_{x_n}.
  \]
\end{thm}

\subsection{Twisted loop algebras of the second kind as equivariant map algebras}\label{subsect: tla2 as ema}
Let $\mathfrak{g}$ be a semisimple Lie algebra over $\mathbb{C}$, $\theta \in \operatorname{Aut}(\mathfrak{g})$ such that $\theta^2 = \mathrm{id}$, and $\mathcal{L}^{\mathrm{tw}} := \mathcal{L}_2(\mathfrak{g},\theta)$ the corresponding twisted loop algebra of the second kind.
Set $R := \mathbb{C}[t,t^{-1}]$, $X := \operatorname{Spec} R$, and $\Gamma := \mathbb{Z}/2\mathbb{Z} = \{ e,\gamma \}$ with identity element $e$.
Let $\Gamma$ act on $R$ by $\gamma t := t^{-1}$, and on $\mathfrak{g}$ by $\gamma x := \theta(x)$.
Then, the equivariant map algebra associated with $(\mathfrak{g},X,\Gamma)$ is isomorphic to $\mathcal{L}^{\mathrm{tw}}$.

We have $X_\mathrm{rat} = \mathbb{C}^\times$, and $\gamma x = x^{-1}$ for all $x \in \mathbb{C}^\times$.
Hence, for each $x \in \mathbb{C}^\times$, we have
\[
  \Gamma_x = \begin{cases}
    \Gamma & \text{ if } x = \pm 1,\\
    \{ e \} & \text{ if } x \neq \pm 1.
  \end{cases}
\]
Consequently,
\[
  \mathfrak{g}^x = \begin{cases}
    \mathfrak{k} & \text{ if } x = \pm 1,\\
    \mathfrak{g} & \text{ if } x \neq \pm 1.
  \end{cases}
\]

Let $P^+$ and $P_\mathfrak{k}^+$ denote the set of dominant integral weights of $\mathfrak{g}$ and $\mathfrak{k}$, respectively (by Proposition \ref{prop: classification involution} \eqref{item: classification involution frk}, $\mathfrak{k}$ is reductive).
For each $\lambda \in P^+$ and $\nu \in P_\mathfrak{k}^+$, let $V(\lambda)$ and $V_\mathfrak{k}(\nu)$ denote the simple $\mathfrak{g}$-module of highest weight $\lambda$ with highest weight vector $v_\lambda$ and the simple $\mathfrak{k}$-module of highest weight $\nu$ with highest weight vector $v_\nu$, respectively.

For each $n \in \mathbb{Z}_{\geq 0}$, set $\mathcal{C}_n := \{ (\alpha_1,\dots,\alpha_n) \in \mathbb{C}^\times \mid \alpha_i \neq \alpha_j^{\pm 1} \ \text{ for all } i \neq j \}$.
Namely, $\mathcal{C}_n$ is the set of $n$-tuples of rational points in $X$ with distinct $\Gamma$-orbits.

For each $\chi \in (\mathcal{L}^{\mathrm{tw}}_\mathrm{ab})^*$, ${\boldsymbol \nu} = (\nu_1,\nu_{-1}) \in (P_\mathfrak{k}^+)^2$, $n \in \mathbb{Z}_{\geq 0}$, ${\boldsymbol \lambda} = ( \lambda_1,\dots,\lambda_n) \in (P^+)^n$, and ${\boldsymbol \alpha} = (\alpha_1,\dots,\alpha_n) \in \mathcal{C}_n$, set
\[
  V(\chi,{\boldsymbol \nu},{\boldsymbol \lambda})_{{\boldsymbol \alpha}} := \mathbb{C}_\chi \otimes V_\mathfrak{k}(\nu_1)_1 \otimes V_\mathfrak{k}(\nu_{-1})_{-1} \otimes V(\lambda_1)_{\alpha_1} \otimes \cdots \otimes V(\lambda_n)_{\alpha_n}.
\]
Also, define $\phi(\chi,{\boldsymbol \nu},{\boldsymbol \lambda})_{{\boldsymbol \alpha}} \in (\mathcal{L}^{\mathrm{tw},0})^*$ by
\begin{align*}
  &\phi(\chi,{\boldsymbol \nu},{\boldsymbol \lambda})_{{\boldsymbol \alpha}}(P(t) \otimes h)
  :=
  \langle \overline{P(t) \otimes h}, \chi \rangle
  +
  \sum_{j=1,-1}P(j) \langle h, \nu_j \rangle
  +
  \sum_{i=1}^n P(\alpha_i) \langle h, \lambda_i \rangle,\\
  &\phi(\chi,{\boldsymbol \nu},{\boldsymbol \lambda})_{{\boldsymbol \alpha}}(Q(t) \otimes h')
  :=
  \langle \overline{Q(t) \otimes h'}, \chi \rangle
  +
  \sum_{i=1}^n Q(\alpha_i) \langle h', \lambda_i \rangle,
\end{align*}
for each $P(t) \in L_+$, $h \in \mathfrak{k}^0$, $Q(t) \in L_-$, $h' \in \mathfrak{p}^0$, where $\overline{\mathfrak{x}} \in \mathcal{L}^{\mathrm{tw}}_\mathrm{ab}$ denotes the image of $\mathfrak{x} \in \mathcal{L}^{\mathrm{tw}}$ under the quotient map.

\begin{prop}
  Let $V$ be a finite-dimensional simple $\mathcal{L}^{\mathrm{tw}}$-module.
  Then, there exist $\chi \in (\mathcal{L}^{\mathrm{tw}}_\mathrm{ab})^*$, ${\boldsymbol \nu} = (\nu_1,\nu_{-1}) \in (P_\mathfrak{k}^+)^2$, $n \in \mathbb{Z}_{\geq 0}$, ${\boldsymbol \lambda} = ( \lambda_1,\dots,\lambda_n) \in (P^+)^n$, and ${\boldsymbol \alpha} = (\alpha_1,\dots,\alpha_n) \in \mathcal{C}_n$ such that
  \[
    V \simeq V(\phi(\chi,{\boldsymbol \nu},{\boldsymbol \lambda})_{{\boldsymbol \alpha}}).
  \]
\end{prop}
\begin{proof}
  By Theorem \ref{thm: classification fd irrep ema}, there exist $\chi,{\boldsymbol \nu},n,{\boldsymbol \lambda},{\boldsymbol \alpha}$ such that $V \simeq V(\chi,{\boldsymbol \nu},{\boldsymbol \lambda})_{\boldsymbol \alpha}$.
  Set
  \[
    v := 1 \otimes v_{\nu_1} \otimes v_{\nu_{-1}} \otimes v_{\lambda_1} \otimes \cdots \otimes v_{\lambda_n} \in V(\chi,{\boldsymbol \nu},{\boldsymbol \lambda})_{\boldsymbol \alpha}.
  \]
  Then, it is straightforwardly verified that $U^\mathrm{tw} v \simeq V(\phi(\chi,{\boldsymbol \nu},{\boldsymbol \lambda})_{\boldsymbol \alpha})$.
  Since $V$ is simple, so is $V(\chi,{\boldsymbol \nu},{\boldsymbol \lambda})_{{\boldsymbol \alpha}}$.
  Hence, we have $U^\mathrm{tw} v = V(\chi,{\boldsymbol \nu},{\boldsymbol \lambda})_{\boldsymbol \alpha}$.
  Summarizing, we obtain
  \[
    V \simeq V(\chi,{\boldsymbol \nu},{\boldsymbol \lambda})_{\boldsymbol \alpha} = U^\mathrm{tw} v \simeq V(\phi(\chi,{\boldsymbol \nu},{\boldsymbol \lambda})_{\boldsymbol \alpha}).
  \]
  Thus, we complete the proof.
\end{proof}

The aim of this paper is to prove this proposition (in a different form) without bypassing the theory of equivariant map algebras.

\section{Preliminaries for the classification}\label{sect: prelim}
This section collects preliminary results that will be used in the next section, where our main result, classification of the finite-dimensional irreducible representation of each twisted loop algebra of the second kind, is proved.

\subsection{Power sum symmetric polynomials}
Let $n \in \mathbb{Z}_{\geq 0}$, and consider the $n$ variables $x_1,\dots,x_n$.
For each $k \in \mathbb{Z}_{> 0}$, let $p_k(x_1,\dots,x_n)$ denote the $k$-th power sum symmetric polynomial in $x$:
\[
  p_k(x_1,\dots,x_n) := \sum_{i=1}^n x_i^k.
\]
Let $z$ be an indeterminate.
Then, the generating function is described as follows:
\[
  P(x_1,\dots,x_n;z) := \sum_{k \geq 1} p_k(x_1,\dots,x_n) z^{k-1} = \sum_{i=1}^n \frac{x_i}{1-x_iz}.
\]

More generally, with another family of variables $y_1,\dots,y_n$, consider the following polynomials:
\[
  p'_k(x_1,\dots,x_n;y_1,\dots,y_n) := \sum_{i=1}^n y_i x_i^{k-1}.
\]
Then, we have
\begin{align}\label{eq: generating function for p'}
  P'(x_1,\dots,x_n;y_1,\dots,y_n;z) := \sum_{k \geq 1} p'_k(x_1,\dots,x_n;y_1,\dots,y_n) z^{k-1} = \sum_{i=1}^n \frac{y_i}{1-x_iz}.
\end{align}

\begin{prop}\label{prop: coincidence of p'}
  Let $m,n \in \mathbb{Z}_{\geq 0}$, $\alpha_i,\beta_i,\gamma_j,\delta_j \in \mathbb{C}$ $(i=1,\dots,m,j=1,\dots,n)$.
  If
  \[
    p'_k(\alpha_1,\dots,\alpha_m;\beta_1,\dots,\beta_m) = p'_k(\gamma_1,\dots,\gamma_n;\delta_1,\dots,\delta_n) \ \text{ for all } k \in \mathbb{Z}_{> 0},
  \]
  then, we have
  \[
    \sum_{\substack{1 \leq i \leq m\\\alpha_i=\epsilon}} \beta_i = \sum_{\substack{1 \leq j \leq n\\\gamma_j=\epsilon}} \delta_j \ \text{ for all } \epsilon \in \mathbb{C}.
  \]
  In particular, if
  \[
    p_k(\alpha_1,\dots,\alpha_m) = p_k(\gamma_1,\dots,\gamma_n) \ \text{ for all } k \in \mathbb{Z}_{> 0},
  \]
  then we have
  \[
    \sharp\{ i \mid \alpha_i = \epsilon \} = \sharp\{ j \mid \gamma_j = \epsilon \} \ \text{ for all } \epsilon \in \mathbb{C}^\times.
  \]
\end{prop}
\begin{proof}
  By equation \eqref{eq: generating function for p'}, we have
  \begin{align}\label{eq: comparision of generating funcitons}
    \sum_{i=1}^m \frac{\beta_i}{1-\alpha_iz} = \sum_{j=1}^n \frac{\delta_j}{1-\gamma_jz}.
  \end{align}
  Let $\epsilon \in \mathbb{C}^\times$.
  Multiplying both sides of equation \eqref{eq: comparision of generating funcitons} by $1-\epsilon z$ and then substituting $z = \epsilon^{-1}$, we obtain
  \[
    \sum_{\substack{1 \leq i \leq m\\\alpha_i=\epsilon}} \beta_i = \sum_{\substack{1 \leq j \leq n\\\gamma_j=\epsilon}} \delta_j.
  \]
  This implies that
  \[
    \sum_{\substack{1 \leq i \leq m\\\alpha_i=0}} \beta_i = \sum_{i=1}^m \frac{\beta_i}{1-\alpha_iz} - \sum_{\substack{1 \leq i \leq m\\\alpha_i \neq 0}} \frac{\beta_i}{1-\alpha_iz} = \sum_{j=1}^n \frac{\delta_j}{1-\gamma_jz} - \sum_{\substack{1 \leq j \leq n\\\gamma_j \neq 0}} \frac{\delta_j}{1-\gamma_jz} = \sum_{\substack{1 \leq j \leq n\\\gamma_j=0}} \delta_j.
  \]
  Hence, we complete the proof of the first assertion.

  For the second assertion, observe that
  \[
    p_k(x_1,\dots,x_n) = p'_k(x_1,\dots,x_n;x_1,\dots,x_n) \text{ for all } k \in \mathbb{Z}_{> 0}.
  \]
  Let us apply the first assertion to the case when $\beta_i = \alpha_i$ for all $i = 1,\dots,m$ and $\delta_j = \gamma_j$ for all $j = 1,\dots,n$.
  Then, for each $\epsilon \in \mathbb{C}^\times$, we obtain
  \[
    \epsilon \cdot \sharp\{ i \mid \alpha_i = \epsilon \} = \sum_{\substack{1 \leq i \leq m \\ \alpha_i = \epsilon}} \beta_i = \sum_{\substack{1 \leq j \leq n \\ \gamma_j = \epsilon}} \delta_j = \epsilon \cdot \sharp\{ j \mid \gamma_j = \epsilon \}.
  \]
  Hence, the assertion follows.
\end{proof}

\subsection{Words}
For each set $A$, let $A^* := \bigsqcup_{r \geq 0} A^r$ denote the set of words with letters in $A$.
It forms a monoid with the concatenation $*$ of words as the multiplication and the empty word $()$ as the identity element.
We often identify an element $a \in A$ and the word $(a) \in A^*$ of length one.
For example,
\[
  a*(a_1,\dots,a_r) = (a)*(a_1,\dots,a_r) = (a,a_1,\dots,a_r).
\]

For each $r \in \mathbb{Z}_{\geq 0}$, let $\mathfrak{S}_r$ denote the $r$-th symmetric group.
It naturally acts on $A^r$:
\[
  \tau \cdot (a_1,\dots,a_r) = (a_{\tau^{-1}(1)},\dots,a_{\tau^{-1}(r)}).
\]

Let $A$ be an additive abelian monoid.
For each $\mathbf{a} = (a_1,\dots,a_r) \in A^*$, define $|\mathbf{a}| \in A$ by
\[
  |\mathbf{a}| := \begin{cases}
    0 & \text{ if } \mathbf{a} = (),\\
    a_1 + \dots + a_r & \text{ if } \mathbf{a} \neq ().
  \end{cases}
\]
For each $r \in \mathbb{Z}_{\geq 0}$, we define the addition on $A^r$ componentwise:
\[
  (a_1,\dots,a_r) + (b_1,\dots,b_r) := (a_1+b_1,\dots,a_r+b_r).
\]

For each $r,k \in \mathbb{Z}_{\geq 0}$, let $P(r,k)$ denote the set of subsequences of length $k$ in the sequence $(1,\dots,r)$:
\[
  P(r,k) := \{ (l_1,\dots,l_r) \mid 1 \leq l_1 < \cdots < l_k \leq r \}.
\]
For each $\mathbf{l} = (l_1,\dots,l_k) \in P(r,k)$, let $(l^1,\dots,l^{r-k})$ denote the complement of $\mathbf{l}$ in $(1,\dots,r)$.
Also, for each $\mathbf{a} = (a_1,\dots,a_r) \in A^r$, set
\[
  \mathbf{a}_\mathbf{l} := (a_{l_1},\dots,a_{l_k}), \quad \mathbf{a}^\mathbf{l} := (a_{l^1},\dots,a_{l^{r-k}}).
\]

For each $\mathbf{l} \in P(r,k)$ and $1 \leq j \leq r-k$, let $i(\mathbf{l},j) \in \{ 1,\dots,k+1\}$ be such that
\[
  l_{i(\mathbf{l},j)-1} < l^j < l_{i(\mathbf{l},j)},
\]
where we set $l_0 = 0$ and $l_{k+1} = r+1$.
Also, set
\[
  \mathbf{l} \circledast l^j := (l_1,\dots,l_{i(\mathbf{l},j)-1},l^j,l_{i(\mathbf{l},j)},\dots,i_k) \in P(r,k+1).
\]
Then, the map
\begin{align}\label{eq: (l,j) to (l*lj,i)}
  P(r,k) \times \{ 1,\dots,r-k \} \to P(r,k+1) \times \{ 1,\dots,k+1 \};\ (\mathbf{l},j) \mapsto (\mathbf{l} \circledast l^j,i(\mathbf{l},j))
\end{align}
is a bijection.
Consequently, the map
\begin{align}\label{eq: (l,j) to l*j}
  P(r,k) \times \{ 1,\dots,r-k \} \to P(r,k+1);\ (\mathbf{l},j) \mapsto \mathbf{l} \circledast l^j
\end{align}
is one-to-$(k+1)$.

For each $r \in \mathbb{Z}_{\geq 0}$, let us identify $\mathfrak{S}_r$ with the subgroup of $\mathfrak{S}_{r+1}$ consisting of permutations fixing $r+1$.
Then, the map
\begin{align}\label{eq: Sr r+1 to Sr+1}
  \mathfrak{S}_r \times \{ 1,\dots,r+1 \} \to \mathfrak{S}_{r+1};\ (\tau,i) \mapsto t_{i,r+1} \tau
\end{align}
is a bijection, where $t_{i,r+1}$ denotes the transposition of $i$ and $r+1$.

\subsection{Auxiliary Lie algebra}\label{sect: aux Lie alg}
In this subsection, we consider an analogue of the loop algebra $\mathcal{L}(\mathfrak{sl}_2)$ and perform similar (but coarser) calculation to \cite[Section 7]{Gar78}, which was used in \cite{Cha86} for classifying the integrable representations of affine Lie algebras.

Let $A$ be an additive abelian monoid.
Also, let $e,f,h$ denote the Chevalley basis of $\mathfrak{sl}_2$, and set $\tilde{\mathcal{L}} := \mathbb{C}[A \times \mathbb{Z}_{\geq 0}] \otimes_\mathbb{C} \mathfrak{sl}_2$.
For each $x \in \mathfrak{sl}_2$ and $(a,n) \in A \times \mathbb{Z}_{\geq 0}$, set
\[
  x_{a,n} := (a,n) \otimes x \in \tilde{\mathcal{L}}.
\]
Then, for each $x,y \in \mathfrak{sl}_2$ and $(a,m),(b,n) \in A \times \mathbb{Z}_{\geq 0}$, we have
\begin{align}\label{eq: mul Ltilsl2}
  [x_{a,m},y_{b,n}] = [x,y]_{a+b,m+n}.
\end{align}

By equation \eqref{eq: mul Ltilsl2}, for each $a \in A$, there exists a unique derivation $D_{a}$ on $\tilde{\mathcal{L}}$ such that
\begin{align}\label{eq: def of Dab}
  D_{a}(x_{b,n}) = n x_{a+b,n+1} \text{ } \text{ for all } x \in \mathfrak{sl}_2, \text{ } (b,n) \in A \times \mathbb{Z}_{\geq 0}.
\end{align}

Let $\tilde{U}$ denote the universal enveloping algebra of $\tilde{\mathcal{L}}$.
For each $a \in A$, the derivation $D_a$ on $\tilde{\mathcal{L}}$ is extended to the derivation, denoted again by $D_a$, on $\tilde{U}$.
Also, let $R_a$ denote the right multiplication on $\tilde{U}$ by $h_{a,1}$.
Finally, set
\[
  \tilde{D}_a := R_a - D_a \in \operatorname{End}_\mathbb{C}(\tilde{U}).
\]

Let $\tilde{U}^{> 0}$ denote the left ideal of $\tilde{U}$ generated by $\{ e_{a,n} \mid (a,n) \in A \times \mathbb{Z}_{\geq 0} \}$, and $\tilde{U}^0$ the subalgebra of $\tilde{U}$ generated by $\{ h_{a,n} \mid (a,n) \in A \times \mathbb{Z}_{\geq 0} \}$.

\begin{lem}\label{lem: tilD preserves tilUpos}
  For each $a \in A$, the linear transformation $\tilde{D}_a$ preserves both $\tilde{U}^{> 0}$ and $\tilde{U}^0$.
\end{lem}
\begin{proof}
  It suffices to show that the linear transformations $D_a$ and $R_a$ preserve both $\tilde{U}^{> 0}$ and $\tilde{U}^0$.
  For $D_a$, since it is a derivation, our claim follows from its definition \eqref{eq: def of Dab}.
  For $R_a$, observe that
  \begin{align*}
    &R_{a}(w) = w h_{a,1} \in \tilde{U}^0,\\
    &R_a(ue_{b,n}) = ue_{b,n}h_{a,1} = u(-2e_{a+b,n+1} + h_{a,1}e_{b,n}) \in \tilde{U}^{>0},
  \end{align*}
  for each $w \in \tilde{U}^0$, $u \in \tilde{U}$, and $(b,n) \in A \times \mathbb{Z}_{\geq 0}$.
  These imply our claim.
  Hence, we complete the proof.
\end{proof}

\begin{lem}\label{lem: tilDa(uv)}
  For each $a \in A$ and $u,v \in \tilde{U}$, we have
  \[
    \tilde{D}_a(uv) = u \tilde{D}_a(v) - D_a(u)v.
  \]
\end{lem}
\begin{proof}
  Let us calculate as follows:
  \begin{align*}
    \tilde{D}_a(uv)
    &= uvh_{a,1} - D_a(uv)\\
    &= u R_a(v) - (D_a(u)v + u D_a(v))\\
    &= u \tilde{D}_a(v) - D_a(u)v.
  \end{align*}
  This completes the proof.
\end{proof}

\begin{lem}\label{lem: comm tilD e}
  For each $a,b \in A$ and $u \in \tilde{U}$, we have
  \[
    \tilde{D}_{a}(e_{b,0} u) = e_{b,0} \tilde{D}_{a}(u).
  \]
\end{lem}
\begin{proof}
  Since $D_a(e_{b,0}) = 0$ by the definition \eqref{eq: def of Dab}, the assertion follows from Lemma \ref{lem: tilDa(uv)}.
\end{proof}

For each $\mathbf{a} = (a_1,\dots,a_r) \in A^*$, set
\begin{align}\label{eq: def of Dbfa, Dtilbfa}
  &D_{\mathbf{a}} := D_{a_1} \cdots D_{a_r}, \quad \tilde{D}_{\mathbf{a}} := \tilde{D}_{a_1} \cdots \tilde{D}_{a_r}.
\end{align}

\begin{lem}
  For each $r \in \mathbb{Z}_{\geq 0}$ and $\mathbf{a} = (a_1,\dots,a_r) \in A^r$, we have
  \[
    \tilde{D}_{\mathbf{a}}(uv) = \sum_{k=0}^r (-1)^k \sum_{\mathbf{l} \in P(r,k)} D_{\mathbf{a}_\mathbf{l}}(u) \tilde{D}_{\mathbf{a}^\mathbf{l}}(v).
  \]
\end{lem}
\begin{proof}
  The assertion can be straightforwardly deduced from Lemma \ref{lem: tilDa(uv)}.
\end{proof}

Let $\equiv_{> 0}$ denote the equivalence relation on $\tilde{U}$ defined by declaring $u \equiv_{>0} v$ to mean that $u-v \in \tilde{U}^{> 0}$.

\begin{lem}\label{lem: er fs}
  For each $r,s \in \mathbb{Z}_{\geq 0}$ such that $r \leq s$, and $\mathbf{a} = (a_1,\dots,a_r) \in A^r$, we have
  \[
    e_{a_1,0} \cdots e_{a_r,0} f_{0,1}^{(s)} \equiv_{>0} \tilde{D}_\mathbf{a}(f_{0,1}^{(s-r)}),
  \]
  where $f_{0,1}^{(n)} := \frac{1}{n!} f_{0,1}^n$ for each $n \in \mathbb{Z}_{\geq 0}$.
\end{lem}
\begin{proof}
  We proceed by induction on $r$.
  The case when $r = 0$ is clear.
  It is straightforwardly verified by induction on $s$ that for each $a \in A$, we have
  \[
    e_{a,0} f_{0,1}^{(s)} = f_{0,1}^{(s-1)} h_{a,1} - f_{0,1}^{(s-2)} f_{a,2} + f_{0,1}^{(s)} e_{a,0} \equiv_{>0} \tilde{D}_{a}(f_{0,1}^{(s-1)}).
  \]
  This proves the assertion for $r = 1$.

  Assume that $r \geq 2$.
  Then, by using Lemmas \ref{lem: tilD preserves tilUpos} and \ref{lem: comm tilD e}, and the induction hypothesis, we can calculate as follows:
  \begin{align*}
    e_{a_1,0} \cdots e_{a_r,0} f_{0,1}^{(s)}
    &= e_{a_r,0} e_{a_1,0} \cdots e_{a_{r-1},0} f_{0,1}^{(s)}\\
    &\equiv_{>0} e_{a_r,0} \tilde{D}_{a_1} \cdots \tilde{D}_{a_{r-1}}(f_{0,1}^{(s-r+1)})\\
    &= \tilde{D}_{a_1} \cdots \tilde{D}_{a_{r-1}}(e_{a_r,0} f_{0,1}^{(s-r+1)})\\
    &\equiv_{>0} \tilde{D}_{a_1} \cdots \tilde{D}_{a_{r-1}}(\tilde{D}_{a_r}(f_{0,1}^{(s-r)}))\\
    &= \tilde{D}_{\mathbf{a}}(f_{0,1}^{(s-r)}).
  \end{align*}
  This completes the proof.
\end{proof}

\begin{prop}\label{prop: er fr+1}
  For each $a \in A$, $r \in \mathbb{Z}_{\geq 0}$, and $\mathbf{a} = (a_1,\dots,a_r) \in A^r$, we have
  \[
    e_{a,0} e_{a_1,0} \cdots e_{a_r,0} f_{0,1}^{(r+1)} \equiv_{>0} \sum_{k=0}^r (-1)^k k! \sum_{\mathbf{l} \in P(r,k)} h_{a+|\mathbf{a}_\mathbf{l}|,k+1} \tilde{D}_{\mathbf{a}^\mathbf{l}}(1).
  \]
\end{prop}
\begin{proof}
  Using Lemmas \ref{lem: er fs} and \ref{lem: tilDa(uv)}, we compute as follows:
  \begin{align*}
    e_{a_1,0} \cdots e_{a_r,0} f_{0,1}^{(r+1)}
    \equiv_{>0}
    \tilde{D}_{\mathbf{a}}(f_{0,1})
    =
    \sum_{k=0}^r (-1)^k \sum_{\mathbf{l} \in P(r,k)} D_{\mathbf{a}_\mathbf{l}}(f_{0,1}) \tilde{D}_{\mathbf{a}^\mathbf{l}}(1).
  \end{align*}
  It is straightforwardly verified that
  \[
    D_{\mathbf{a}_\mathbf{l}}(f_{0,1}) = k! f_{|\mathbf{a}_\mathbf{l}|,k+1}.
  \]
  Hence, we obtain
  \[
    e_{a_1,0} \cdots e_{a_r,0} f_{0,1}^{(r+1)} \equiv_{>0} \sum_{k=0}^r (-1)^k k! \sum_{\mathbf{l} \in P(r,k)} f_{|\mathbf{a}_\mathbf{l}|,k+1} \tilde{D}_{\mathbf{a}^\mathbf{l}}(1).
  \]
  Noting that $e_{a,0} f_{|\mathbf{a}_\mathbf{l}|,k+1} = h_{a+|\mathbf{a}_\mathbf{l}|,k+1} + f_{|\mathbf{a}_\mathbf{l}|,k+1} e_{a,0}$ and $e_{a,0} \tilde{D}_{\mathbf{a}^\mathbf{l}}(1) \in \tilde{U}^{>0}$, we see that the assertion follows.
\end{proof}

\begin{cor}\label{cor: tilD ab*pi}
  For each $a \in A$, $r \in \mathbb{Z}_{\geq 0}$, and $\mathbf{a} = (a_1,\dots,a_r) \in A^*$, we have
  \[
    \tilde{D}_{a*\mathbf{a}}(1) = \sum_{k=0}^r (-1)^k k! \sum_{\mathbf{l} \in P(r,k)} h_{a+|\mathbf{a}_\mathbf{l}|,k+1} \tilde{D}_{\mathbf{a}^\mathbf{l}}(1).
  \]
\end{cor}
\begin{proof}
  By Lemma \ref{lem: er fs} and Proposition \ref{prop: er fr+1}, we have
  \begin{align*}
    \tilde{D}_{a*\mathbf{a}}(1) \equiv_{>0} \sum_{k=0}^r (-1)^k k! \sum_{\mathbf{l} \in P(r,k)} h_{a+|\mathbf{a}_\mathbf{l}|,k+1} \tilde{D}_{\mathbf{a}^\mathbf{l}}(1).
  \end{align*}
  Since the both sides are elements of $\tilde{U}^0$, they coincide with each other.
\end{proof}

Let us summarize properties of $\tilde{D}_\mathbf{a}(1)$'s.

\begin{prop}\label{prop: properties of Dtilbfa(1)}
  For each $a \in A$, $r \in \mathbb{Z}_{\geq 0}$, and $\mathbf{a} = (a_1,\dots,a_r) \in A^r$, the following hold.
  \begin{enumerate}
    \item\label{item: Dtilempty(1)} $\tilde{D}_{()}(1) = 1.$
    \item\label{item: Dtila(1)} $\tilde{D}_a(1) = h_{a,1}$.
    \item\label{item: Dtilbfas(1) equiv ef} $\tilde{D}_\mathbf{a}(1) \equiv_{>0} e_{a_1,0} \cdots e_{a_r,0} f_{(0,1)}^{(r)}$.
    \item\label{item: Dtiltaubfa(1)} $\tilde{D}_{\tau \mathbf{a}}(1) = \tilde{D}_\mathbf{a}(1)$.
    \item\label{item: Dtilabfa(1)} $\tilde{D}_{a*\mathbf{a}}(1) = \sum_{k=0}^r (-1)^k k! \sum_{\mathbf{l} \in P(r,k)} h_{a+|\mathbf{a}_\mathbf{l}|,k+1} \tilde{D}_{\mathbf{a}^\mathbf{l}}(1)$.
  \end{enumerate}
\end{prop}
\begin{proof}
  The assertions \eqref{item: Dtilempty(1)} and \eqref{item: Dtila(1)} are immediate from the definition \eqref{eq: def of Dbfa, Dtilbfa} of $\tilde{D}_\mathbf{a}$.
  The assertion \eqref{item: Dtilbfas(1) equiv ef} is a special case ($s = r$) of Lemma \ref{lem: er fs}.
  This implies \eqref{item: Dtiltaubfa(1)} since the elements $e_{a_1,0},\dots,e_{a_r,0}$ commute with each other.
  The assertion \eqref{item: Dtilabfa(1)} is just Corollary \ref{cor: tilD ab*pi}.
\end{proof}

\subsection{Invariant rings}
Let $A$ be a finitely generated additive abelian monoid which is equipped with a total order $\leq$ such that $a \leq b$ implies $a+c \leq b+c$ for all $a,b,c \in A$.
Then, the monoid algebra $\mathbb{C}[A]$ is an affine algebra, i.e., a finitely generated integral domain.
Hence, the set $X := \operatorname{Specm} \mathbb{C}[A]$ of maximal ideals of $\mathbb{C}[A]$ is an affine variety.

\begin{ex}
  Suppose that $A = \mathbb{Z}$ with usual addition and order.
  Then, the monoid algebra $\mathbb{C}[A]$ and the affine variety $X$ are identical to the Laurent polynomial ring $\mathbb{C}[x,x^{-1}]$ and the algebraic torus $\mathbb{C}^\times$, respectively.
\end{ex}

\begin{ex}
  Let $m$ be a positive integer.
  Suppose that $A = \mathbb{Z}_{\geq 0}^m$ with the componentwise addition and the lexicographic order.
  Then, the monoid algebra $\mathbb{C}[A]$ and the affine variety $X$ are identical to the polynomial ring $\mathbb{C}[x_1,\dots,x_m]$ and the affine space $\mathbb{C}^m$, respectively.
\end{ex}

Let $n$ be a positive integer, and consider the direct product $X^n$.
It is identified with $\operatorname{Specm} \mathbb{C}[A^n]$.

The $n$-th symmetric group $\mathfrak{S}_n$ naturally acts on $X^n$.
Then, the orbit space $X^n/\mathfrak{S}_n$ can be identified with $\operatorname{Specm} \mathbb{C}[A^n]^{\mathfrak{S}_n}$ in a way such that the quotient map $X^n \to X^n/\mathfrak{S}_n$ corresponds to the inclusion map $\mathbb{C}[A^n]^{\mathfrak{S}_n} \to \mathbb{C}[A^n]$.

For each $i \in \{ 1,\dots,n \}$ and $a \in A$, let $x_i^a \in \mathbb{C}[A^n]$ denote the element corresponding to $(0,\dots,0,\overset{i}{a},0,\dots,0) \in A^n$.
For each $\mathbf{a} = (a_1,\dots,a_n) \in A^n$, set $\mathbf{x}^\mathbf{a} := x_1^{a_1} \cdots x_n^{a_n}$.
For each $r \in \mathbb{Z}_{\geq 0}$ and $\mathbf{a} = (a_1,\dots,a_r) \in A^r$, define $m_\mathbf{a} \in \mathbb{C}[A^n]^{\mathfrak{S}_n}$ by
\begin{align}\label{eq: def of m_bfa}
  m_\mathbf{a} := \sum_{\tau \in \mathfrak{S}_r} \sum_{1 \leq i_1 < \dots < i_r \leq n} x_{i_1}^{a_{\tau^{-1}(1)}} \cdots x_{i_r}^{a_{\tau^{-1}(r)}}.
\end{align}
In particular,
\[
  m_a = \sum_{i=1}^n x_i^a \ \text{ for all } a \in A.
\]

\begin{rem}\label{rem: da are monomial symm poly}
  When $r \leq n$, we have
  \[
    m_\mathbf{a} = \frac{1}{(n-r)!} \sum_{\tau \in \mathfrak{S}_n} \mathbf{x}^{\tau \cdot (a_1,\dots,a_r,0,\dots,0)} = \frac{r!}{|\mathfrak{S}_r \mathbf{a}|} \sum_{\mathbf{b} \in \mathfrak{S}_n (a_1,\dots,a_r,0,\dots,0)} \mathbf{x}^\mathbf{b}.
  \]
  Hence, the functions $m_\mathbf{a}$ are analogues of the monomial symmetric polynomials.
\end{rem}

\begin{prop}\label{prop: basic properties of m_bfa}
  For each $r \in \mathbb{Z}_{\geq 0}$ and $\mathbf{a} = (a_1,\dots,a_r) \in A^r$, the following hold.
  \begin{enumerate}
    \item \label{item: m_()=1}$m_{()} = 1$.
    \item $m_0 = n$.
    \item $m_\mathbf{a} = 0$ if $r > n$.
    \item \label{item: m_taubfa = m_bfa}$m_{\tau \mathbf{a}} = m_\mathbf{a}$ for all $\tau \in \mathfrak{S}_r$.
    \item \label{item: m_a m_bfa}$m_a m_\mathbf{a} = m_{a*\mathbf{a}} + \sum_{l=1}^r m_{(a_1,\dots,a_l+a,\dots,a_r)}$ for all $a \in A$.
    \item \label{item: m_a1 dots m_ar}$m_{a_1} \cdots m_{a_r} = \sum_{k=1}^r \frac{1}{k!} \sum_{\substack{r_1,\dots,r_k \in \mathbb{Z}_{> 0}\\r_1+\dots+r_k = r}} \sum_{\substack{(\mathbf{l}_1,\dots,\mathbf{l}_k) \in P(r,r_1) \times \cdots \times P(r,r_k)\\ \mathbf{l}_1 \sqcup \cdots \sqcup \mathbf{l}_k = \{ 1,\dots,r \}}} m_{(|\mathbf{a}_{\mathbf{l}_1}|,\dots,|\mathbf{a}_{\mathbf{l}_k}|)}$.
  \end{enumerate}
\end{prop}
\begin{proof}
  The assertions \eqref{item: m_()=1}--\eqref{item: m_taubfa = m_bfa} are clear from the definition \eqref{eq: def of m_bfa}.

  Let us prove the assertion \eqref{item: m_a m_bfa}.
  By the definition of $m_\mathbf{a}$, we have
  \begin{align*}
    m_a m_\mathbf{a}
    &=
    \sum_{\tau \in \mathfrak{S}_r} \sum_{1 \leq i_1 < \cdots < i_r \leq n} \sum_{i=1}^n x_i^a x_{i_1}^{a_{\tau^{-1}(1)}} \cdots x_{i_r}^{a_{\tau^{-1}(r)}}\\
    &=
    \sum_{\tau \in \mathfrak{S}_r} \sum_{\mathbf{l} \in P(n,r)} \sum_{j=1}^{r} x_{l_{\tau(j)}}^a x_{l_1}^{a_{\tau^{-1}(1)}} \cdots x_{l_r}^{a_{\tau^{-1}(r)}}\\
    &\phantom{=}+
    \sum_{\tau \in \mathfrak{S}_r} \sum_{\mathbf{l} \in P(n,r)} \sum_{j=1}^{n-r} x_{l^j}^a x_{i_1}^{a_{\tau^{-1}(1)}} \cdots x_{i_r}^{a_{\tau^{-1}(r)}}.
  \end{align*}
  The first term of the right-hand side equals $\sum_{l=1}^r m_{(a_1,\dots,a_l+a,\dots,a_r)}$.
  For the second term, consider the bijection
  \[
    \mathfrak{S}_r \times \{1,\dots,n-r\} \to \mathfrak{S}_{r+1} \times P(n,r+1);\ (\tau,\mathbf{l},j) \mapsto (\sigma_{i(\mathbf{l},j),r+1}\tau, \mathbf{l} \circledast l^j)
  \]
  obtained by composing the bijections \eqref{eq: (l,j) to (l*lj,i)} and \eqref{eq: Sr r+1 to Sr+1}.
  Then, we see that the second term coincides with $m_{a*\mathbf{a}}$.

  The assertion \eqref{item: m_a1 dots m_ar} can be straightforwardly deduced from the assertion \eqref{item: m_a m_bfa}.
\end{proof}

\begin{defi}
  Let $r \in \mathbb{Z}_{\geq 0}$ and $\mathbf{a} = (a_1,\dots,a_r) \in A^r$.
  We say that $\mathbf{a}$ is \emph{$n$-dominant} if the following are satisfied:
  \begin{itemize}
    \item $r \leq n$,
    \item $a_1,\dots,a_r \neq 0$,
    \item $a_1 \geq \cdots \geq a_r$.
  \end{itemize}
  Let $A_n^+$ denote the set of $n$-dominant elements in $A^*$.
\end{defi}

\begin{prop}\label{prop: basis of inv ring}
  The set $\{ m_\mathbf{a} \mid \mathbf{a} \in A_n^+ \}$ forms a basis of $\mathbb{C}[A^n]^{\mathfrak{S}_n}$.
\end{prop}
\begin{proof}
  Taking Remark \ref{rem: da are monomial symm poly} into account, one can prove the assertion in a similar way to \cite[Section 2]{Mac79}, where it is proved that the monomial symmetric polynomials form a basis of the ring of symmetric polynomials.
\end{proof}

\begin{thm}\label{thm: presentation of inv ring}
  The algebra $\mathbb{C}[A^n]^{\mathfrak{S}_n}$ has the following presentation$:$
  \begin{description}
    \item[Generators] $\{ m_\mathbf{a} \mid \mathbf{a} \in A^* \}$.
    \item[Relations] For each $r \in \mathbb{Z}_{\geq 0}$ and $\mathbf{a} = (a_1,\dots,a_r) \in A^r$,
    \begin{align}\label{eq: def rel of inv ring}
      \begin{split}
        &m_{()} = 1,\\
        &m_0 = n,\\
        &m_\mathbf{a} = 0 \ \text{ if } r > n,\\
        &m_a m_\mathbf{a} = m_{a*\mathbf{a}} + \sum_{l=1}^r m_{(a_1,\dots,a_l+a,\dots,a_r)} \ \text{ for all } a \in A.
      \end{split}
    \end{align}
  \end{description}
\end{thm}
\begin{proof}
  Let $R$ denote the $\mathbb{C}$-algebra generated by $\{ m'_\mathbf{a} \mid \mathbf{a} \in A^* \}$ subject to the relations \eqref{eq: def rel of inv ring} with $m$ replaced by $m'$.
  Let us show that $R$ is spanned by
  \begin{align}\label{eq: spanning set for R}
    \{ m'_{\mathbf{a}} \mid \mathbf{a} \in A_n^+ \}.
  \end{align}
  To this end, it suffices to prove the following for all $r \in \mathbb{Z}_{\geq 0}$ and $\mathbf{a} \in A^r$:
  \begin{itemize}
    \item $m'_{0*\mathbf{a}} = (n-r) m'_\mathbf{a}$,
    \item $m'_{\tau \mathbf{a}} = m'_\mathbf{a}$ for all $\tau \in \mathfrak{S}_r$.
  \end{itemize}
  The first claim follows from the relations \eqref{eq: def rel of inv ring}:
  \[
    n m'_\mathbf{a} = m'_0 m'_\mathbf{a} = m'_{0*\mathbf{a}} + rm'_\mathbf{a}.
  \]
  For the second claim, observe that the $m'_\mathbf{a}$'s satisfy the same identity as Proposition \ref{prop: basic properties of m_bfa} \eqref{item: m_a1 dots m_ar}. 
  Then, the claim can be deduced by induction on $r$.

  By Proposition \ref{prop: basic properties of m_bfa}, there exists an algebra homomorphism $f : R \to \mathbb{C}[A^n]^{\mathfrak{S}_n}$ which sends $m'_\mathbf{a}$ to $m_\mathbf{a}$ for all $\mathbf{a} \in A^*$.
  By Proposition \ref{prop: basis of inv ring} and the fact that the set \eqref{eq: spanning set for R} spans $R$, we see that this homomorphism is an isomorphism.
\end{proof}

\begin{thm}\label{thm: sufficient condition for evaluation}
  Let $(c_\mathbf{a})_{\mathbf{a} \in A^*}$ be a family of complex numbers satisfying the following for all $r \in \mathbb{Z}_{\geq 0}$ and $\mathbf{a} = (a_1,\dots,a_r) \in A^r$$:$
  \begin{align*}
    &c_{()} = 1,\\
    &c_0 = n,\\
    &c_\mathbf{a} = 0 \ \text{ if } r > n,\\
    &c_a c_\mathbf{a} = c_{a*\mathbf{a}} + \sum_{l=1}^r c_{(a_1,\dots,a_l+a,\dots,a_r)} \ \text{ for all } a \in A.
  \end{align*}
Then, there exists $\mathbf{p} \in X^n$ such that $m_\mathbf{a}(\mathbf{p}) = c_\mathbf{a}$ for all $\mathbf{a} \in A^*$, where $m_\mathbf{a}(\mathbf{p})$ denotes the value of the function $m_\mathbf{a}$ at $\mathbf{p}$ $($recall that $m_\mathbf{a} \in \mathbb{C}[A^n]^{\mathfrak{S}_n} \subset \mathbb{C}[A^n]$ is a function on $X^n$$)$.
\end{thm}
\begin{proof}
  By Theorem \ref{thm: presentation of inv ring}, there exists an algebra homomorphism $f : \mathbb{C}[A^n]^{\mathfrak{S}_n} \to \mathbb{C}$ such that
  \[
    f(m_\mathbf{a}) = c_\mathbf{a} \ \text{ for all } \mathbf{a} \in A^*.
  \]
  Let $f^* : \{(0)\} \to X^n/\mathfrak{S}_n$ denote the morphism of affine varieties corresponding to $f$, where $(0) \subset \mathbb{C}$ denotes the zero ideal.
  Hence, there exists $\mathbf{p} \in X^n$ whose $\mathfrak{S}_n$-orbit is $f^*((0))$.
  Let $\tilde{f} : \mathbb{C}[A^n] \to \mathbb{C}$ denote the algebra homomorphism corresponding to the morphism $\{(0)\} \to X^n;\ (0) \mapsto \mathbf{p}$ of affine varieties.
  Then, we have $\tilde{f}|_{\mathbb{C}[A^n]^{\mathfrak{S}_n}} = f$.
  In particular, $\tilde{f}(m_\mathbf{a}) = f(m_\mathbf{a}) = c_\mathbf{a}$ for all $\mathbf{a} \in A^*$.
  Thus, we complete the proof.
\end{proof}

\begin{thm}\label{thm: sufficient condition for evaluation2}
  Suppose that $A \setminus \{0\}$ is closed under the addition.
  Let $(c_\mathbf{a})_{\mathbf{a} \in (A \setminus \{0\})^*}$ be a family of complex numbers satisfying the following for all $r \in \mathbb{Z}_{\geq 0}$ and $\mathbf{a} = (a_1,\dots,a_r) \in (A \setminus \{0\})^r$$:$
  \begin{align*}
    &c_{()} = 1,\\
    &c_\mathbf{a} = 0 \ \text{ if } r > n,\\
    &c_a c_\mathbf{a} = c_{a*\mathbf{a}} + \sum_{l=1}^r c_{(a_1,\dots,a_l+a,\dots,a_r)} \ \text{ for all } a \in A \setminus \{0\}.
  \end{align*}
  Then, there exists $\mathbf{p} \in X^n$ such that $m_\mathbf{a}(\mathbf{p}) = c_\mathbf{a}$ for all $\mathbf{a} \in (A \setminus \{0\})^*$.
\end{thm}
\begin{proof}
  For each $\mathbf{a} = (a_1,\dots,a_r) \in (A \setminus \{0\})^*$, $s \in \mathbb{Z}_{\geq 0}$, and $\mathbf{b} \in \mathfrak{S}_{r+s} (a_1,\dots,a_r,0,\dots,0)$, set
  \[
    c_{\mathbf{b}} := (n-r)(n-r-1) \cdots (n-r-s+1) c_\mathbf{a}.
  \]
  In particular,
  \[
    c_0 = n c_{()} = n.
  \]
  It is straightforwardly verified that
  \[
    c_a c_\mathbf{a} = c_{a*\mathbf{a}} + \sum_{l=1}^r c_{(a_1,\dots,a_l+a,\dots,a_r)} \ \text{ for all } a \in A,\ r \in \mathbb{Z}_{\geq 0},\ \mathbf{a} \in A^r.
  \]
  Then, we can apply Theorem \ref{thm: sufficient condition for evaluation} to obtain the assertion.
\end{proof}

\section{Classification of finite-dimensional irreducible representations}\label{sect: classification}
In this section, we classify the finite-dimensional irreducible representations of each twisted loop algebra of the second kind.
In the first four subsections below, we treat four minimal twisted loop algebras of the second kind one by one.

In this section, given a Dynkin diagram $I$ of finite type, $\mu \in \operatorname{Aut}(I)$ such that $\mu^2 = \mathrm{id}$, and $S \subseteq I^\mu$, let $\mathfrak{g}$ denote the corresponding semisimple Lie algebra, $\theta \in \operatorname{Aut}(\mathfrak{g})$ the automorphism given by equation \eqref{eq: theta(I,sigma,J)}, and $\mathcal{L}^{\mathrm{tw}} := \mathcal{L}^{\mathrm{tw}}(I,\mu,S)$.
Also, let $\mathfrak{k}$ and $\mathfrak{p}$ denote the eigenspaces of $\theta$ of eigenvalue $1$ and $-1$, respectively.
The notations $P^+$, $P_\mathfrak{k}^+$, $V(\lambda)$, $V_\mathfrak{k}(\lambda)$, $v_\lambda$, $v_\nu$ are as in Subsection \ref{subsect: tla2 as ema}.
Given a finitely generated additive abelian monoid $A$ equipped with a compatible total order, set $X := \operatorname{mSpec} \mathbb{C}[A]$.
Let $\tilde{\mathcal{L}}$ denote the Lie algebra in Subsection \ref{sect: aux Lie alg}, and $\tilde{U}$ the universal enveloping algebra of $\tilde{\mathcal{L}}$.
Set $t_\pm := t \pm t^{-1} \in \mathbb{C}[t,t^{-1}]$.

\subsection{Type $\Delta A_1$}
In this subsection, we consider the twisted loop algebra of the second kind associated with $(I,\mu,\emptyset)$, where
$I = \{1,2\}$ denotes the Dynkin diagram of type $\Delta A_1 := A_1 \times A_1$ and
$\mu \in \operatorname{Aut}(I) \simeq \mathbb{Z}/2\mathbb{Z}$ the nontrivial automorphism.
Let $A$ denote the additive abelian monoid $\mathbb{Z}$ with usual addition and ordering
Then, we have
\begin{itemize}
  \item $\mathfrak{g} = \mathfrak{sl}_2 \oplus \mathfrak{sl}_2$,
  \item $\mathfrak{k} = \langle e_1+e_2,f_1+f_2,h_1+h_2 \rangle \simeq \mathfrak{sl}_2$,
  \item $\mathfrak{p} = \operatorname{Span}\{ e_1-e_2,f_1-f_2,h_1-h_2 \}$,
  \item $P^+ = \mathbb{Z}_{\geq 0}^2$,
  \item $P_\mathfrak{k}^+ = \mathbb{Z}_{\geq 0}$,
  \item $X = \mathbb{C}^\times$.
\end{itemize}

For each $a \in \mathbb{Z}$, define the following elements in $\mathcal{L}$:
\[
  x_a := t^a \otimes e_1 + t^{-a} \otimes e_2, \quad y_a := t^a \otimes f_1 + t^{-a} \otimes f_2, \quad w_a := t^a \otimes h_1 + t^{-a} \otimes h_2.
\]
Then, the set
\[
  \{ x_a,y_a,w_a \mid a \in \mathbb{Z} \}
\]
forms a basis of $\mathcal{L}^\mathrm{tw}$, and the following hold for each $a,a' \in \mathbb{Z}$:
\begin{align}\label{eq: mul table dA1}
  \begin{split}
    &[x_a, x_{a'}] = [y_a, y_{a'}] = [w_a, w_{a'}] = 0,\\
    &[x_a, y_{a'}] = w_{a+a'},\quad [w_a, x_{a'}] = 2x_{a+a'}, \quad [w_a, y_{a'}] = -2y_{a+a'}.
  \end{split}
\end{align}
Hence, $\mathcal{L}^{\mathrm{tw}}$ is isomorphic to the loop algebra $L \otimes \mathfrak{sl}_2$.

\begin{lem}
  There exists an algebra homomorphism
  \[
    \mathrm{ev} : \tilde{U} \to U^{\mathrm{tw}}
  \]
  which sends $e_{a,n}$, $f_{a,n}$, $h_{a,n}$ to $x_a$, $y_a$, $w_a$, respectively for each $(a,n) \in A \times \mathbb{Z}_{\geq 0}$.
\end{lem}
\begin{proof}
  The assertion follows from equations \eqref{eq: mul Ltilsl2} and \eqref{eq: mul table dA1}.
\end{proof}

For each $\mathbf{a} \in A^*$, set
\[
  d_\mathbf{a} := \mathrm{ev}(\tilde{D}_\mathbf{a}(1)) \in U^{\mathrm{tw}, 0}.
\]

\begin{prop}\label{prop: properties of dbfa DeltaA1}
  For each $a \in A$, $r \in \mathbb{Z}_{\geq 0}$, and $\mathbf{a} = (a_1,\dots,a_r) \in A^r$, the following hold$:$
  \begin{enumerate}
    \item $d_{()} = 1$.
    \item $d_a = w_a$.
    \item\label{item: dbfa equiv xy DeltaA1} $d_\mathbf{a} \equiv_{> 0} x_{a_1} \cdots x_{a_r} y_0^{(r)}$.
    \item\label{item: dtaubfa = dbfa DeltaA1} $d_{\tau\mathbf{a}} = d_\mathbf{a}$ for all $\tau \in \mathfrak{S}_r$.
    \item\label{item: dabfa DeltaA1} $d_{a*\mathbf{a}} = \sum_{k=0}^r (-1)^k k! \sum_{\mathbf{l} \in P(r,k)} d_{a+|\mathbf{a}_\mathbf{l}|} d_{\mathbf{a}^\mathbf{l}}$.
    \item $d_a d_\mathbf{a} = d_{a*\mathbf{a}} + \sum_{l=1}^r d_{(a_1,\dots,a_l+a,\dots,a_r)}$.
  \end{enumerate}
\end{prop}
\begin{proof}
  Since $\mathrm{ev} : \tilde{U} \to U^\mathrm{tw}$ is an algebra homomorphism, the first five assertions follow from Proposition \ref{prop: properties of Dtilbfa(1)}.
  Let us prove the last one by induction on $r$.
  The case when $r = 0$ is clear.
  Hence, assume that $r \geq 1$.
  By assertions \eqref{item: dtaubfa = dbfa DeltaA1} and \eqref{item: dabfa DeltaA1}, and the induction hypothesis, we have
  \begin{align*}
    d_{a*\mathbf{a}}
    &=
    \sum_{k=0}^r (-1)^k k! \sum_{\mathbf{l} \in P(r,k)} d_{a+|\mathbf{a}_\mathbf{l}|} d_{\mathbf{a}^\mathbf{l}}\\
    &=
    d_{a} d_\mathbf{a}\\
    &\phantom{=} + \sum_{k=1}^{r-1} (-1)^k k! \sum_{\mathbf{l} \in P(r,k)} (
      d_{(a+|\mathbf{a}_\mathbf{l}|)*\mathbf{a}^\mathbf{l}}
      + \sum_{j=1}^{r-k} d_{(a+|\mathbf{a}_\mathbf{l}|+a_{l^j})*(\mathbf{a}^{\mathbf{l} \circledast l^j})}
    )\\
    &\phantom{=} + (-1)^r r! d_{a+|\mathbf{a}|},
  \end{align*}
  By the map \eqref{eq: (l,j) to l*j}, this can be simplified as
  \[
    d_{a*\mathbf{a}} = d_{a} d_{\mathbf{a}} - \sum_{l=1}^r d_{(a+a_l)*\mathbf{a}^l}.
  \]
  Hence, the assertion follows.
\end{proof}

\begin{thm}\label{thm: classification DeltaA1}
  For each $\phi \in (\mathcal{L}^{\mathrm{tw},0})^*$, the following are equivalent:
  \begin{enumerate}
    \item $\dim V(\phi) < \infty$.
    \item There exist $n \in \mathbb{Z}_{\geq 0}$ and $\alpha_1,\dots,\alpha_n \in \mathbb{C}^\times$ such that
    \[
      \phi(w_a) = \sum_{i=1}^n \alpha_i^a \ \text{ for each } a \in \mathbb{Z}.
    \]
  \end{enumerate}
\end{thm}
\begin{proof}
  First, assume the condition $(2)$, and consider the tensor product module
  \begin{align}\label{eq: V(1)a1 ot V(1)ar}
    V := V(1,0)_{\alpha_1} \otimes \cdots \otimes V(1,0)_{\alpha_l}.
  \end{align}
  Then, the submodule of $V$ generated by $v_{1,0} \otimes \cdots \otimes v_{1,0}$ is a highest weight module of highest weight $\phi$.
  This implies that $V(\phi)$ is a quotient of $V$, and hence, finite-dimensional.

  Next, assume the condition (1).
  Set
  \[
    r_0 := \max\{ r \geq 0 \mid y_0^{(r)} v_\phi \neq 0 \}.
  \]
  By Proposition \ref{prop: properties of dbfa DeltaA1} \eqref{item: dbfa equiv xy DeltaA1}, we have
  \begin{align}\label{eq: phi(dbfa)=0 DeltaA1}
    \phi(d_\mathbf{a}) = 0 \ \text{ for all } r > r_0,\ \mathbf{a} \in A^r.
  \end{align}
  Also, by the representation theory of $\mathfrak{sl}_2$, we have
  \begin{align}\label{eq: phi(d0) DeltaA1}
    \phi(d_0) = \phi(w_0) = r_0.
  \end{align}

  For each $\mathbf{a} \in A^*$, set
  \[
    c_\mathbf{a} := \phi(d_\mathbf{a}) \in \mathbb{C}.
  \]
  By Proposition \ref{prop: properties of dbfa DeltaA1}, equations \eqref{eq: phi(dbfa)=0 DeltaA1} and \eqref{eq: phi(d0) DeltaA1}, and Theorem \ref{thm: sufficient condition for evaluation}, there exist $\alpha_1,\dots,\alpha_{r_0} \in \mathbb{C}^\times$ such that
  \[
    c_\mathbf{a} = m_\mathbf{a}(\alpha_1,\dots,\alpha_{r_0}) \ \text{ for all } \mathbf{a} \in A^*.
  \]
  In particular, we obtain
  \[
    \phi(w_a) = \phi(d_a) = c_a = m_a(\alpha_1,\dots,\alpha_{r_0}) = \sum_{i=1}^{r_0} \alpha_i^a \ \text{ for all } a \in \mathbb{Z}.
  \]
  Thus, we complete the proof.
\end{proof}

\begin{rem}\label{rem: reform DeltaA1}
  Set $w_\pm := e_1 \pm e_2$, and for each $a \in \mathbb{Z}_{\geq 0}$, set
  \[
    w_{+,a} := t_+^a \otimes w_+, \quad w_{-,a} := t_+^a t_- \otimes w_-.
  \]
  Let $\phi \in (\mathcal{L}^{\mathrm{tw},0})^*$ be as in Theorem \ref{thm: classification DeltaA1} (2).
  Then, we have
  \begin{align*}
    &\phi(w_{+,a}) = \sum_{i=1}^n (\alpha_i + \alpha_i^{-1})^a, \\
    &\phi(w_{-,a}) = \sum_{i=1}^n (\alpha_i + \alpha_i^{-1})^a (\alpha_i - \alpha_i^{-1})
  \end{align*}
  for all $a \in \mathbb{Z}_{\geq 0}$.
\end{rem}

\subsection{Type $A_1$}
In this subsection, we consider the twisted loop algebra of the second kind associated with $(I,\mathrm{id},\emptyset)$, where
$I$ denotes the Dynkin diagram of type $A_1$.
Let $A$ denote the additive abelian monoid $\mathbb{Z}_{\geq 0}$ with usual addition and ordering.
Then, we have
\begin{itemize}
  \item $\mathfrak{g} = \mathfrak{sl}_2$,
  \item $\mathfrak{k} = \mathfrak{sl}_2$,
  \item $\mathfrak{p} = 0$,
  \item $P^+ = \mathbb{Z}_{\geq 0}$,
  \item $P_\mathfrak{k}^+ = \mathbb{Z}_{\geq 0}$,
  \item $X = \mathbb{C}$.
\end{itemize}

For each $a \in \mathbb{Z}_{\geq 0}$, set
\[
  x_a := t_+^a \otimes e, \quad y_a := t_+^a \otimes f, \quad w_a := t_+^a \otimes h.
\]
Then, the set
\[
  \{ x_a,y_a,w_a \mid a \in \mathbb{Z}_{\geq 0} \}
\]
forms a basis of $\mathcal{L}^\mathrm{tw}$, and the following hold for each $a,a' \in \mathbb{Z}_{\geq 0}$:
\begin{align}\label{eq: mul table A1}
  \begin{split}
    &[x_a,e_{a'}] = [y_a,y_{a'}] = [w_a,w_{a'}] = 0,\\
    &[x_a,y_{a'}] = w_{a+a'},
    \quad
    [w_a,x_{a'}] = 2x_{a+a'},
    \quad
    [w_a,y_{a'}] = -2y_{a+a'}.
  \end{split}
\end{align}
Hence, the Lie algebra $\mathcal{L}^{\mathrm{tw}}$ is isomorphic to the current algebra $\mathbb{C}[t] \otimes \mathfrak{sl}_2$.

\begin{lem}
  There exists an algebra homomorphism
  \[
    \mathrm{ev} : \tilde{U} \to U^{\mathrm{tw}}
  \]
  which sends $e_{a,n}$, $f_{a,n}$, $h_{a,n}$ to $x_a$, $y_a$, $w_a$, respectively for each $(a,n) \in A \times \mathbb{Z}_{\geq 0}$.
\end{lem}
\begin{proof}
  The assertion follows from equations \eqref{eq: mul Ltilsl2} and \eqref{eq: mul table A1}.
\end{proof}

For each $\mathbf{a} \in A^*$, set
\[
  d_\mathbf{a} := \mathrm{ev}(\tilde{D}_\mathbf{a}(1)) \in U^{\mathrm{tw}, 0}.
\]

\begin{thm}\label{thm: classification A1}
  For each $\phi \in (\mathcal{L}^{\mathrm{tw},0})^*$, the following are equivalent:
  \begin{enumerate}
    \item $\dim V(\phi) < \infty$.
    \item There exist $n \in \mathbb{Z}_{\geq 0}$ and $\alpha_1,\dots,\alpha_n \in \mathbb{C}^\times$ such that
    \[
      \phi(w_a) = \sum_{i=1}^n (\alpha_i+\alpha_i^{-1})^a \ \text{ for each } a \geq 0.
    \]
  \end{enumerate}
\end{thm}
\begin{proof}
  First, assume the condition $(2)$, and consider the tensor product module
  \[
    V := V(1)_{\alpha_1} \otimes \cdots \otimes V(1)_{\alpha_n}.
  \]
  Then, the submodule of $V$ generated by $v_1 \otimes \dots \otimes v_1$ is a highest weight module of highest weight $\phi$.
  This implies that $V(\phi)$ is a quotient of $V$, and hence, finite-dimensional.

  Next, assume the condition (1).
  Set
  \[
    r_0 := \max\{ r \geq 0 \mid y_0^{(r)} v_\phi \neq 0 \}.
  \]
  By a similar way to the proof of Theorem \ref{thm: classification DeltaA1}, we see that there exist $\beta_1,\dots,\beta_{r_0} \in \mathbb{C}$ such that
  \[
    \phi(w_a) = m_a(\beta_1,\dots,\beta_{r_0}) = \sum_{i=1}^{r_0} \beta_i^a \ \text{ for all } a \geq 0.
  \]
  For each $i = 1,\dots,r_0$, there exists $\alpha_i \in \mathbb{C}^\times$ such that $\beta_i = \alpha_i + \alpha_i^{-1}$.
  Hence, the assertion follows.
\end{proof}

\subsection{Type $A\mathrm{I}_1$}
In this subsection, we consider the twisted loop algebra of the second kind associated with $(I,\mathrm{id},I)$, where
$I$ denotes the Dynkin diagram of type $A_1$.
Let $A$ denote the additive abelian monoid $\mathbb{Z}_{\geq 0}^2$ with the componentwise addition and the lexicographic order.
Then, we have
\begin{itemize}
  \item $\mathfrak{g} = \mathfrak{sl}_2$,
  \item $\mathfrak{k} = \mathbb{C} h \simeq \mathfrak{so}_2$,
  \item $\mathfrak{p} = \mathbb{C} e \oplus \mathbb{C} f$,
  \item $P^+ = \mathbb{Z}_{\geq 0}$,
  \item $P_\mathfrak{k}^+ = \mathbb{C}$,
  \item $X = \mathbb{C}^2$.
\end{itemize}

For each $a,b \in \mathbb{Z}_{\geq 0}$, set
\[
  x_{a,b} := t_+^a t_-^b \otimes e, \quad y_{a,b} := t_+^a t_-^b \otimes f, \quad w_{a,b} := t_+^a t_-^b \otimes h.
\]
Then, the following hold for each $a,b,a',b' \in \mathbb{Z}_{\geq 0}$:
\begin{align}\label{eq: mul table AI1}
  \begin{split}
    &[x_{a,b},e_{a',b'}] = [y_{a,b},y_{a',b'}] = [w_{a,b},w_{a',b'}] = 0,\\
    &[x_{a,b},y_{a',b'}] = w_{a+a',b+b'},
    \quad
    [w_{a,b},x_{a',b'}] = 2x_{a+a',b+b'},
    \quad
    [w_{a,b},y_{a',b'}] = -2y_{a+a',b+b'}.
  \end{split}
\end{align}

The set
\[
  \{ x_{a,1}, y_{a,1}, w_{a,0} \mid a \in \mathbb{Z}_{\geq 0} \}
\]
forms a basis of $\mathcal{L}^\mathrm{tw}$.

\begin{rem}
  The Lie algebra $\mathcal{L}^{\mathrm{tw}}$ is isomorphic to the Onsager algebra.
\end{rem}

\begin{lem}
  There exists an algebra homomorphism
  \[
    \mathrm{ev} : \tilde{U} \to U
  \]
  which sends $e_{(a,b),n}$, $f_{(a,b),n}$, $h_{(a,b),n}$ to $x_{a,b+n}$, $y_{a,b+n}$, $w_{a,b+n}$, respectively for each $((a,b),n) \in A \times \mathbb{Z}_{\geq 0}$.
\end{lem}
\begin{proof}
  The assertion follows from equations \eqref{eq: mul Ltilsl2} and \eqref{eq: mul table AI1}.
\end{proof}

Let $\mathbb{Z}_{> 0, \mathrm{odd}}$ denote the set of positive odd integers, and set $B := \mathbb{Z}_{\geq 0} \times \mathbb{Z}_{> 0,\mathrm{odd}}$.
For each $\mathbf{a} \in B^*$, set
\[
  d_\mathbf{a} := \mathrm{ev}(\tilde{D}_\mathbf{a}(1)) \in U^{\mathrm{tw}, 0}.
\]

\begin{prop}\label{prop: properties of dbfa AI1}
  For each $(a,b) \in B$, $r \in \mathbb{Z}_{\geq 0}$, $\mathbf{a} = ((a_1,b_1),\dots,(a_r,b_r)) \in B^r$, the following hold.
  \begin{enumerate}
    \item $d_{()} = 1$.
    \item $d_{(a,b)} = w_{a,b+1}$.
    \item\label{item: dbfa equiv xy AI1} $d_\mathbf{a} \equiv_{> 0} x_{a_1,b_1} \cdots x_{a_r,b_r} y_{0,1}^{(r)}$.
    \item $d_{\tau \mathbf{a}} = d_\mathbf{a}$ for all $\tau \in \mathfrak{S}_r$.
    \item $d_{(a,b)*\mathbf{a}} = \sum_{k=0}^r (-1)^k k! \sum_{\mathbf{l} \in P(r,k)} d_{(a,b)+|\mathbf{a}_\mathbf{l}|+k} d_{\mathbf{a}^\mathbf{l}}$.
    \item $d_{(a,b)} d_\mathbf{a} = d_{(a,b)*\mathbf{a}} + \sum_{l=1}^r d_{((a_1,b_1),\dots,(a_l+a,b_l+b),\dots,(a_r,b_r))}$.
  \end{enumerate}
\end{prop}
\begin{proof}
  The assertion can be proved in a similar way to Proposition \ref{prop: properties of dbfa DeltaA1}.
\end{proof}

We will frequently use the following formula:

\begin{lem}\label{lem: binom formula}
  Let $a,b,c \in \mathbb{Z}_{\geq 0}$.
  Then, we have
  \[
    w_{a+2c,b} = \sum_{j=0}^c \binom{c}{j} 4^{c-j} w_{a,b+2j}.
  \]
\end{lem}
\begin{proof}
  Since $t_+^2 = t_-^2+4$, we have
  \[
    t_+^{a+2c} t_-^b = t_+^a(t_-^2+4)^ct_-^b = \sum_{j=0}^c \binom{c}{j} 4^{c-j} t_+^at_-^{b+2j}.
  \]
  This implies the assertion.
\end{proof}

\begin{thm}\label{thm: classification AI1}
  For each $\phi \in (\mathcal{L}^{\mathrm{tw},0})^*$, the following are equivalent:
  \begin{enumerate}
    \item $\dim V(\phi) < \infty$.
    \item There exist $\nu_1,\nu_{-1} \in \mathbb{C}$, $n \in \mathbb{Z}_{\geq 0}$, and $\alpha_1,\dots,\alpha_n \in \mathbb{C}^\times \setminus \{ \pm 1 \}$ such that
    \[
      \phi(w_{a,0}) = \nu_1 \cdot 2^a + \nu_{-1} \cdot (-2)^a + \sum_{i=1}^n (\alpha_i+\alpha_i^{-1})^a \ \text{ for all } a \geq 0.
    \]
  \end{enumerate}
\end{thm}
\begin{proof}
  First, assume the condition $(2)$, and consider the tensor product module
  \[
    V := V_\mathfrak{k}(\nu_1)_1 \otimes V_\mathfrak{k}(\nu_{-1})_{-1} \otimes V(1)_{\alpha_1} \otimes \cdots \otimes V(1)_{\alpha_n}.
  \]
  Then, the submodule of $V$ generated by $v_{\nu_1} \otimes v_{\nu_{-1}} \otimes v_1 \otimes \dots \otimes v_1$ is a highest weight module of highest weight $\phi$.
  This implies that $V(\phi)$ is a quotient of $V$, and hence, finite-dimensional.

  Next, assume the condition $(1)$.
  Set
  \[
    r_0 := \max\{ r \geq 0 \mid y_{0,1}^{(r)} v_\phi \neq 0 \}.
  \]
  By Proposition \ref{prop: properties of dbfa AI1} \eqref{item: dbfa equiv xy DeltaA1}, we have
  \begin{align}\label{eq: phi(dbfa)=0 AI1}
    \phi(d_\mathbf{a}) = 0 \ \text{ for all } r > r_0 \text{ and } \mathbf{a} \in B^r.
  \end{align}
  For each $\mathbf{a} = ((a_1,b_1),\dots,(a_r,b_r)) \in (A \setminus \{(0,0)\})^*$, set
  \[
    c_\mathbf{a} := \phi(d_{((a_1,2(a_1+b_1)-1),\dots,(a_r,2(a_r+b_r)-1))}) \in \mathbb{C}.
  \]
  By Proposition \ref{prop: properties of dbfa AI1}, equation \eqref{eq: phi(dbfa)=0 AI1}, and Theorem \ref{thm: sufficient condition for evaluation2}, there exist $\delta_1,\dots,\delta_{r_0},\gamma_1,\dots,\gamma_{r_0} \in \mathbb{C}$ such that
  \[
    c_\mathbf{a} = m_\mathbf{a}((\delta_1,\gamma_1),\dots,(\delta_{r_0},\gamma_{r_0})) \ \text{ for all } \mathbf{a} \in (A \setminus \{(0,0)\})^*.
  \]
  In particular, we have
  \[
    \phi(w_{a,2(a+b)}) = \phi(d_{(a,2(a+b)-1)}) = c_{(a,b)} = \sum_{i=1}^{r_0} \delta_i^a \gamma_i^{b} \ \text{ for all } (a,b) \in A \setminus \{(0,0)\}.
  \]

  Note that we may freely rearrange the tuple $((\delta_1,\gamma_1),\dots,(\delta_{r_0},\gamma_{r_0}))$.
  For convenience, we set $r_1 := \sharp \{ i \mid \gamma_i \neq 0 \}$ and assume that $\gamma_i = 0$ for all $i > r_1$.

  For each $k \in \mathbb{Z}_{> 0}$ and $l \in \mathbb{Z}_{\geq 0}$, using Lemma \ref{lem: binom formula}, we obtain
  \[
    \sum_{i=1}^{r_0} \delta_i^{2k} \gamma_i^l = c_{(2k,l)} = \phi(w_{2k,4k+2l}) = \sum_{i=1}^{r_0} (\gamma_i+4)^k \gamma_i^{2k+l}.
  \]
  This implies that
  \[
    p'_{l+1}(\gamma_1,\dots,\gamma_{r_0};\delta_1^{2k},\dots,\delta_{r_0}^{2k}) = p'_{l+1}(\gamma_1,\dots,\gamma_{r_0};\tilde{\beta}_1^{k},\dots,\tilde{\beta}_{r_0}^{k}) \ \text{ for all } l \geq 0,
  \]
  where $\tilde{\beta}_i := (\gamma_i+4)\gamma_i^2$.
  By Proposition \ref{prop: coincidence of p'}, we have
  \[
    \sum_{\substack{1 \leq i \leq r_0\\\gamma_i=\epsilon}} \delta_i^{2k} = \sum_{\substack{1 \leq i \leq r_0\\\gamma_i=\epsilon}} \tilde{\beta_i}^k \ \text{ for all } k \geq 1, \ \epsilon \in \mathbb{C}.
  \]
  Again by Proposition \ref{prop: coincidence of p'}, for each $\gamma$, we obtain
  \[
    \sharp\{ i \mid \gamma_i = \epsilon_1 \text{ and } \delta_i^2 = \epsilon_2 \} = \sharp\{ i \mid \gamma_i = \epsilon_1 \text{ and } \tilde{\beta}_i = \epsilon_2 \} \ \text{ for all } \epsilon_1 \in \mathbb{C}, \text{ } \epsilon_2 \in \mathbb{C}^\times.
  \]
  By the definition of $\tilde{\beta}_i$, we have
  \[
    \sharp\{ i \mid \gamma_i = \epsilon_1 \text{ and } \tilde{\beta}_i = \epsilon_2 \} = \begin{cases}
      \sharp\{ i \mid \gamma_i = \epsilon_1 \} & \text{ if } \epsilon_2 = (\epsilon_1+4)\epsilon_1^2, \\
      0 & \text{ if } \epsilon_2 \neq (\epsilon_1+4)\epsilon_1^2.
    \end{cases}
  \]
  Therefore, we see that
  \[
    \delta_i^2 = (\gamma_i+4)\gamma_i^2 \ \text{ for all } i = 1,\dots,r_0.
  \]
  
  Set $\beta_i := \delta_i \gamma_i^{-1}$ for each $i = 1,\dots,r_1$.
  Then, we have
  \begin{align}\label{eq: beta = gamma + 4 neq 4}
    \beta_i^2 = \gamma_i + 4 \neq 4 \ \text{ for all } i = 1,\dots,r_0.
  \end{align}

  For each $k \in \mathbb{Z}_{> 0}$, we have
  \begin{align*}
    &\phi(w_{0,2k}) = c_{(0,k)} = \sum_{i=1}^{r_0} \gamma_i^k = \sum_{i=1}^{r_1} \gamma_i^k,\\
    &\phi(w_{1,2k}) = c_{(1,k-1)} = \sum_{i=1}^{r_0} \delta_i \gamma_i^{k-1} = \sum_{i=1}^{r_1} \beta_i \gamma_i^{k}
  \end{align*}
  Then, for each $a \in \mathbb{Z}_{> 0}$, we obtain
  \begin{align*}
    \phi(w_{2a,0})
    &= \phi(\sum_{j=0}^a \binom{a}{j} 4^{a-j} w_{0,2j})\\
    &= \sum_{j=0}^a \binom{a}{j} 4^{a-j} \sum_{i=1}^{r_1} \gamma_i^j + 4^a(\phi(w_{0,0})-r_1)\\
    &= \sum_{i=1}^{r_1} (\gamma_i+4)^a + 4^a(\phi(w_{0,0})-r_1)\\
    &= \sum_{i=1}^{r_1} \beta_i^{2a} + 4^a(\phi(w_{0,0})-r_1),
  \end{align*}
  and
  \begin{align*}
    \phi(w_{2a+1,0})
    &= \phi(\sum_{j=0}^a \binom{a}{j} 4^{a-j} w_{1,2j})\\
    &= \sum_{j=0}^a \binom{a}{j} 4^{a-j} \sum_{i=1}^{r_1} \beta_i \gamma_i^j + 4^a(\phi(w_{1,0})-\sum_{i=1}^{r_1} \beta_i)\\
    &= \sum_{i=1}^{r_1} \beta_i (\gamma_i+4)^a + 4^a(\phi(w_{1,0})-\sum_{i=1}^{r_1} \beta_i)\\
    &= \sum_{i=1}^{r_1} \beta_i^{2a+1} + 4^a(\phi(w_{1,0})-\sum_{i=1}^{r_1} \beta_i).
  \end{align*}
  Define $\nu_1,\nu_{-1} \in \mathbb{C}$ by
  \[
    \nu_1+\nu_{-1} = \phi(w_{0,0})-r_1, \quad \nu_1-\nu_{-1} = \frac{1}{2}(\phi(w_{1,0})-\sum_{i=1}^{r_1} \beta_i).
  \]
  Then, we have
  \[
    \phi(w_{a,0}) = \nu_1 \cdot 2^a + \nu_{-1} \cdot (-2)^a + \sum_{i=1}^{r_1} \beta_i^a \ \text{ for all } a \in \mathbb{Z}_{\geq 0}.
  \]
  For each $i = 1,\dots,r_1$, there exists $\alpha_i \in \mathbb{C}^\times$ such that $\beta_i = \alpha_i + \alpha_i^{-1}$.
  The condition \eqref{eq: beta = gamma + 4 neq 4} implies that $\alpha_i \neq \pm 1$ for all $i = 1,\dots,r_1$.
  Thus, we complete the proof.
\end{proof}

\subsection{Type $A\mathrm{I}_2$}
In this subsection, we consider the twisted loop algebra of the second kind associated with $(I,\mu,\emptyset)$, where
$I$ denotes the Dynkin diagram of type $A_2$ and $\mu \in \operatorname{Aut}(I) \simeq \mathbb{Z}/2\mathbb{Z}$ the nontrivial automorphism.
Let $A$ denote the additive abelian monoid $\mathbb{Z}_{\geq 0}^2$ with the componentwise addition and the lexicographic order.
Then, we have
\begin{itemize}
  \item $\mathfrak{g} = \mathfrak{sl}_3$,
  \item $\mathfrak{k} = \operatorname{Span}\{ x_+,y_+,w_+ \} \simeq \mathfrak{so}_3$,
  \item $\mathfrak{p} = \operatorname{Span}\{ X,x_-,w_-,y_-,Y \}$,
  \item $P^+ = \{ (a,b) \in \mathbb{Z}_{\geq 0}^2 \mid a \geq b \}$,
  \item $P_\mathfrak{k}^+ = \frac{1}{2} \mathbb{Z}_{\geq 0}$,
  \item $X = \mathbb{C}^2$,
\end{itemize}
where
\[
  Y := [f_1,f_2], \ y_\pm := f_1 \pm f_2, \ w_\pm := h_1 \pm h_2, \ x_\pm := e_1 \pm e_2, \ X := [e_2,e_1].
\]

For each $(a,b) \in A$ and $x \in \mathfrak{g}$, set
\[
  x_{(a,b)} = x_{a,b} := t_+^a t_-^b \otimes x \in \mathcal{L}.
\]

For each $\mathbf{a} \in A^*$, define $d_\mathbf{a} \in U^{\mathrm{tw}, 0}$ as follows:
\begin{itemize}
  \item $d_{()} = 1$,
  \item $d_{(a,b)} = \begin{cases}
    w_{+,a,b} & \text{ if $b$ is even},\\
    -w_{-,a,b} & \text{ if $b$ is odd},
  \end{cases}$
  \item $d_{(a,b)*\mathbf{a}} = d_{(a,b)}d_\mathbf{a} - \sum_{l=1}^r d_{((a_1,b_1),\dots,(a_l+a,b_l+b),\dots,(a_r,b_r))}$.
\end{itemize}

Set
\[
  u_{-1} := 2y_+, \quad u_0 := 2w_+, \quad u_1 := -x_+
\]
and
\[
  v_{-2} := 4Y, \quad v_{-1} := -4y_-, \quad v_0 := -2w_-, \quad v_1 := 2x_-, \quad v_2 := X.
\]
Then, for each $r \in \mathbb{Z}_{\geq 0}$, we have
\[
  x_+^{(r)} u_{-1} = u_{r-1}, \quad x_+^{(r)} v_{-2} = v_{r-2},
\]
where we set $u_s = 0$ if $s  > 1$ and $v_s = 0$ if $s > 2$.

For each $\mathbf{a} = ((a_1,b_1),\dots,(a_r,b_r)) \in A^r$, set
\[
  v_{-2,\mathbf{a}} := v_{-2,(a_1,b_1)} \cdots v_{-2,(a_r,b_r)} \in U.
\]

\begin{lem}\label{lem: x2r+2 v-2b AI2}
  For each $(a,b) \in A$, $r \in \mathbb{Z}_{\geq 0}$, and $\mathbf{a} = ((a_1,b_1),\dots,(a_r,b_r)) \in A^r$, we have
  \begin{align*}
    x_{+,0,0}^{(2r+2)} v_{-2,(a,b)*\mathbf{a}}
    &\equiv_{>0} v_{0,(a,b)} x_{+,0,0}^{(2r)} v_{-2,\mathbf{a}}\\
    &\phantom{\equiv} -2 \sum_{l=1}^r x_{+,0,0}^{(2r-2)} v_{-2,\mathbf{a}^l} u_{0,|\mathbf{a}_l|+(a,b)}\\
    &\phantom{\equiv} + 8\sum_{1 \leq l_1 < l_2 \leq r} x_{+,0,0}^{(2r-2)} v_{-2,(|\mathbf{a}_{(l_1,l_2)}|+(a,b))*\mathbf{a}^{(l_1,l_2)}}.
  \end{align*}
\end{lem}
\begin{proof}
  We have
  \begin{align}\label{eq: x2r v-2b der}
    x_{+,0,0}^{(2r+2)} v_{-2,(a,b)*\mathbf{a}} \equiv_{>0} v_{0,(a,b)} x_{+,0,0}^{(2r)} v_{-2,\mathbf{a}} + v_{1,(a,b)} x_{+,0,0}^{(2r-1)} v_{-2,\mathbf{a}} + v_{2,(a,b)} x_{+,0,0}^{(2r-2)} v_{-2,\mathbf{a}}.
  \end{align}
  By the identities $[x_+,v_2] = 0$, $[v_2,v_{-2}] = 2u_0$, and $[u_0,v_{-2}]=-4v_{-2}$, the third term of the right-hand side is computed as follows:
  \begin{align}\label{eq: 3-1 AI2}
    \begin{split}
      v_{2,(a,b)} x_{+,0,0}^{(2r-2)} v_{-2,\mathbf{a}}
      &\equiv_{>0} x_{+,0,0}^{(2r-2)} \sum_{l=1}^r v_{-2,(a_1,b_1)} \cdots 2u_{0,(a_l+a,b_l+b)} \cdots v_{-2,(a_r,b_r)}\\
      &\equiv_{>0} 2 x_{+,0,0}^{(2r-2)}(\sum_{l=1}^r v_{-2,\mathbf{a}^l} u_{0,|\mathbf{a}_l|+(a,b)} - 4\sum_{1 \leq l_1 < l_2 \leq r} v_{-2,(|\mathbf{a}_{(l_1,l_2)}+(a,b)|)*\mathbf{a}^{(l_1,l_2)}}).
    \end{split}
  \end{align}
  On the other hand, noting that
  \begin{align*}
    x_{+,0,0}^{(r)} u_{0,(a,b)} = 2 x_{+,0,0}^{(r-1)}u_{1,(a,b)} + u_{0,(a,b)}x_{+,0,0}^{(r)},
  \end{align*}
  we obtain
  \begin{align}\label{eq: 3-2 AI2}
    \begin{split}
      &x_{+,0,0}^{(2r-2)} \sum_{l=1}^r v_{-2,(a_1,b_1)} \cdots 2u_{0,(a_l+a,b_l+b)} \cdots v_{-2,(a_r,b_r)}\\
      &\quad\equiv_{>0} 2 x_{+,0,0}^{(2r-2)}(\sum_{l=1}^r u_{0,|\mathbf{a}_l|+(a,b)} v_{-2,\mathbf{a}^l} + 4\sum_{1 \leq l_1 < l_2 \leq r} v_{-2,(|\mathbf{a}_{(l_1,l_2)}|+(a,b))*\mathbf{a}^{(l_1,l_2)}})\\
      &\quad\equiv_{>0} 2\sum_{l=1}^r(2 x_{+,0,0}^{(2r-3)} u_{1,|\mathbf{a}_l|+(a,b)} + u_{0,|\mathbf{a}_l|+(a,b)} x_{+,0,0}^{(2r-2)}) v_{-2,\mathbf{a}^l}\\
      &\quad\phantom{\equiv} + 8\sum_{1 \leq l_1 < l_2 \leq r} x_{+,0,0}^{(2r-2)} v_{-2,(|\mathbf{a}_{(l_1,l_2)}|+(a,b))*\mathbf{a}^{(l_1,l_2)}}.
    \end{split}
  \end{align}
  Equations \eqref{eq: 3-1 AI2} and \eqref{eq: 3-2 AI2} imply that
  \begin{align}\label{eq: 3-3 AI2}
    \sum_{l=1}^r x_{+,0,0}^{(2r-3)} u_{1,|\mathbf{a}_l|+(a,b)} v_{-2,\mathbf{a}^l} = -4 \sum_{1 \leq l_1 < l_2 \leq r} x_{+,0,0}^{(2r-2)} v_{-2,(|\mathbf{a}_{(l_1,l_2)}|+(a,b))*\mathbf{a}^{(l_1,l_2)}}.
  \end{align}

  Next, let us consider the second term of the right-hand side of equation \eqref{eq: x2r v-2b der}.
  Noting that,
  \begin{align}\label{eq: 3-4 AI2}
    v_{1,(a,b)} x_{+,0,0}^{(2r-1)} = x_{+,0,0}^{(2r-1)} v_{1,(a,b)} -4v_{2,(a,b)} x_{+,0,0}^{(2r-2)},
  \end{align}
  and
  \begin{align*}
    x_{+,0,0}^{(2r-1)} u_{-1,(a,b)} = x_{+,0,0}^{(2r-3)} u_{1,(a,b)} + u_{0,(a,b)} x_{+,0,0}^{(2r-2)} + u_{-1,(a,b)}x_{+,0,0}^{(2r-1)},
  \end{align*}
  we obtain
  \begin{align}\label{eq: 3-5 AI2}
    \begin{split}
      x_{+,0,0}^{(2r-1)} v_{1,(a,b)} v_{-2,\mathbf{a}}
      &\equiv_{>0} x_{+,0,0}^{(2r-1)} \sum_{l=1}^r v_{-2,(a_1,b_1)} \cdots u_{-1,(a_l+a,b_l+b)} \cdots v_{-2,(a_r,b_r)}\\
      &\equiv_{>0} 4x_{+,0,0}^{(2r-1)} \sum_{l=1}^r u_{-1,|\mathbf{a}_l|+(a,b)} v_{-2,\mathbf{a}^l}\\
      &\equiv_{>0} 4\sum_{l=1}^r (x_{+,0,0}^{(2r-3)} u_{1,|\mathbf{a}_l|+(a,b)} + u_{0,|\mathbf{a}_l|+(a,b)} x_{+,0,0}^{(2r-2)}) v_{-2,\mathbf{a}^l}.
    \end{split}
  \end{align}
  
  Finally, let us compute as follows:
  \begin{align*}
    x_{+,0,0}^{(2r+2)} v_{-2,(a,b)*\mathbf{a}}
    \overset{\eqref{eq: x2r v-2b der}}{\equiv_{>0}}
    &v_{0,(a,b)} x_{+,0,0}^{(2r)} v_{-2,\mathbf{a}} + v_{1,(a,b)} x_{+,0,0}^{(2r-1)} v_{-2,\mathbf{a}} + v_{2,(a,b)} x_{+,0,0}^{(2r-2)} v_{-2,\mathbf{a}}\\
    \overset{\eqref{eq: 3-4 AI2}}{=}
    &v_{0,(a,b)} x_{+,0,0}^{(2r)} v_{-2,\mathbf{a}} + x_{+,0,0}^{(2r-1)} v_{1,(a,b)} v_{-2,\mathbf{a}} - 3v_{2,(a,b)} x_{+,0,0}^{(2r-2)} v_{-2,\mathbf{a}}\\
    \overset{\eqref{eq: 3-1 AI2}}{\equiv_{>0}}
    &v_{0,(a,b)} x_{+,0,0}^{(2r)} v_{-2,\mathbf{a}} + x_{+,0,0}^{(2r-1)} v_{1,(a,b)} v_{-2,\mathbf{a}}\\
    \phantom{\equiv_{>0}}
    &-6\sum_{l=1}^r x_{+,0,0}^{(2r-2)} v_{-2,\mathbf{a}^l} u_{0,|\mathbf{a}_l|+(a,b)}\\
    \phantom{\equiv_{>0}}
    &+ 24\sum_{1 \leq l_1 < l_2 \leq r} x_{+,0,0}^{(2r-2)} v_{-2,(|\mathbf{a}_{(l_1,l_2)}+(a,b)|)*\mathbf{a}^{(l_1,l_2)}}\\
    \overset{\eqref{eq: 3-3 AI2}, \eqref{eq: 3-5 AI2}}{\equiv_{>0}}
    &v_{0,(a,b)} x_{+,0,0}^{(2r)} v_{-2,\mathbf{a}}\\
    \phantom{\equiv_{>0}}
    &-2\sum_{l=1}^r x_{+,0,0}^{(2r-2)} v_{-2,\mathbf{a}^l} u_{0,|\mathbf{a}_l|+(a,b)}\\
    \phantom{\equiv_{>0}}
    &+ 8\sum_{1 \leq l_1 < l_2 \leq r} x_{+,0,0}^{(2r-2)} v_{-2,(|\mathbf{a}_{(l_1,l_2)}+(a,b)|)*\mathbf{a}^{(l_1,l_2)}}.
  \end{align*}
  Thus, we complete the proof.
\end{proof}

Set $B := \mathbb{Z}_{\geq 0} \times \mathbb{Z}_{> 0,\mathrm{odd}}$.

\begin{lem}\label{lem: x2r+2 v-2b equiv dbfa AI2}
  For each $r \in \mathbb{Z}_{\geq 0}$ and $\mathbf{a} = (a_1,\dots,a_r) \in B^r$, we have
  \[
    2^r x_{+,0,0}^{(2r)} Y_{(a_1,b_1)} \cdots Y_{(a_r,b_r)} \equiv_{>0} d_\mathbf{a}.
  \]
  and
  \[
    d_{\tau \mathbf{a}} = d_\mathbf{a} \ \text{ for all } \tau \in \mathfrak{S}_r.
  \]
\end{lem}
\begin{proof}
  The assertions can be straightforwardly deduced by Lemma \ref{lem: x2r+2 v-2b AI2} and induction on $r$.
\end{proof}

\begin{lem}\label{lem: deduction from A1}
  Let $\phi \in (\mathcal{L}^{\mathrm{tw},0})^*$ be such that the highest weight module $V(\phi)$ is finite-dimensional.
  Set
  \[
    n := \max\{ r \geq 0 \mid y_{+,0,0}^{(r)} v_\phi \neq 0 \}.
  \]
  Then, there exist $\beta_1,\dots,\beta_n \in \mathbb{C}$ such that
  \[
    \phi(2w_{+,a,0}) = \sum_{i=1}^n \beta_i^a \ \text{ for all } a \in \mathbb{Z}_{\geq 0}.
  \]
\end{lem}
\begin{proof}
  It is easily verified that the elements $x_+,2y_+,2w_+$ forms an $\mathfrak{sl}_2$-triple.
  Then, the assertion follows from Theorem \ref{thm: classification A1}.
\end{proof}

\begin{lem}\label{lem: deduction from AI1}
  Let $\phi \in (\mathcal{L}^{\mathrm{tw},0})^*$ be such that the highest weight module $V(\phi)$ is finite-dimensional.
  Set
  \[
    n := \max\{ r \geq 0 \mid Y_{0,1}^{(r)} v_\phi \neq 0 \}.
  \]
  Then, the following hold$:$
  \begin{enumerate}
    \item There exist $\nu_1,\nu_{-1} \in \mathbb{C}$ and $\beta_1,\dots,\beta_n \in \mathbb{C}^\times \setminus \{ \pm 2 \}$ such that
    \[
      \phi(w_{+,a,0}) = \nu_1 \cdot 2^a + \nu_{-1} \cdot (-2)^a + \sum_{i=1}^n (\beta_i+\beta_i^{-1})^a \ \text{ for all } a \in \mathbb{Z}_{\geq 0}.
    \]
    \item $\phi(d_\mathbf{a}) = 0$ for all $r > n$ and $\mathbf{a} \in B^r$.
  \end{enumerate}
\end{lem}
\begin{proof}
  The first assertion follows from Theorem \ref{thm: classification AI1} in a similar way to the proof of Lemma \ref{lem: deduction from A1}.
  For the second assertion, observe that Theorem \ref{thm: classification AI1} states that $V(\phi)$ is a quotient of a module on which the product $Y_{a_1,b_1} \cdots Y_{a_r,b_r}$ acts as zero for all $r > n$ and $((a_1,b_1),\dots,(a_r,b_r)) \in B^r$.
  This, together with Lemma \ref{lem: x2r+2 v-2b equiv dbfa AI2}, proves the assertion.
\end{proof}

\begin{thm}\label{thm: classification AI2}
  For each $\phi \in (\mathcal{L}^{\mathrm{tw},0})^*$, the following are equivalent$:$
  \begin{enumerate}
    \item $\dim V(\phi) < \infty$.
    \item There exist $\nu_1,\nu_{-1} \in \frac{1}{2}\mathbb{Z}_{\geq 0}$, $n \in \mathbb{Z}_{\geq 0}$, and $\alpha_1,\dots,\alpha_r \in \mathbb{C}^\times \setminus \{ \pm 1 \}$ such that
    \begin{align*}
      &\phi(w_{+,a,0}) = \nu_1 \cdot 2^a + \nu_{-1} \cdot (-2)^a + \sum_{i=1}^n (\alpha_i+\alpha_i^{-1})^a,\\
      &\phi(w_{-,a,1}) = \sum_{i=1}^n (\alpha_i+\alpha_i^{-1})^a (\alpha_i-\alpha_i^{-1}).
    \end{align*}
  \end{enumerate}
\end{thm}
\begin{proof}
  First, assume the condition $(2)$, and consider the tensor product module
  \[
    V := V_\mathfrak{k}(\nu_1)_1 \otimes V_\mathfrak{k}(\nu_{-1})_{-1} \otimes V(1)_{\alpha_1} \otimes \cdots \otimes V(1)_{\alpha_n}.
  \]
  Then, the submodule of $V$ generated by $v_{\nu_1} \otimes v_{\nu_{-1}} \otimes v_1 \otimes \dots \otimes v_1$ is a highest weight module of highest weight $\phi$.
  This implies that $V(\phi)$ is a quotient of $V$, and hence, finite-dimensional.

  Next, assume the condition $(1)$.
  Set
  \[
    r_0 := \max\{ r \geq 0 \mid y_{+,0,0}^{(r)} v_\phi \neq 0 \}.
  \]
  By Lemma \ref{lem: deduction from A1}, there exist $\beta'_1,\dots,\beta'_{r_0} \in \mathbb{C}$ such that
  \begin{align}\label{eq: phi(2w+a0) from A1}
    \phi(2w_{+,a,0}) = \sum_{i=1}^{r_0} (\beta'_i)^a \ \text{ for all } a \in \mathbb{Z}_{\geq 0}.
  \end{align}

  Set
  \[
    r_1 := \max\{ r \geq 0 \mid Y_{0,1}^{(r)} v_\phi \neq 0 \}.
  \]
  By Lemma \ref{lem: deduction from AI1}, there exist $\nu_{\pm 1} \in \mathbb{C}$ and $\beta''_1,\dots,\beta''_{r_1} \in \mathbb{C}^\times \setminus \{ \pm 2 \}$ such that
  \begin{align}\label{eq: phi(w+a0) from AI1}
    \phi(w_{+,a,0}) = \nu_1 \cdot 2^a + \nu_{-1} \cdot (-2)^a + \sum_{i=1}^{r_1} (\beta''_i)^a \ \text{ for all } a \in \mathbb{Z}_{\geq 0}.
  \end{align}

  Equations \eqref{eq: phi(2w+a0) from A1} and \eqref{eq: phi(w+a0) from AI1} imply that
  \[
    p'_k(\beta'_1,\dots,\beta'_{r_0};\beta'_1,\dots,\beta'_{r_0}) = p'_k(2,-2,\beta''_1,\dots,\beta''_{r_1};4\nu_1,-4\nu_{-1},2\beta''_1,\dots,2\beta''_{r_1}) \ \text{ for all } k \geq 1.
  \]
  By Proposition \ref{prop: coincidence of p'}, we obtain
  \begin{align*}
    &\pm 2 \cdot \sharp\{ i \mid \beta'_i = \pm 2 \} = \pm 4\nu_{\pm 1},\\
    &\epsilon \cdot \sharp\{ i \mid \beta'_i = \epsilon \} = 2\epsilon \cdot \sharp\{ i \mid \beta''_i = \epsilon \} \ \text{ for all } \epsilon \in \mathbb{C}^\times \setminus \{ \pm 2 \}.
  \end{align*}
  In particular,
  \[
    \nu_{\pm 1} \in \frac{1}{2}\mathbb{Z}_{\geq 0}.
  \]

  By Lemma \ref{lem: deduction from AI1}, we have
  \begin{align}\label{eq: phi(dbfa) = 0 AI2}
    \phi(d_\mathbf{a}) = 0 \ \text{ for all } r > r_1,\ \mathbf{a} \in B^r.
  \end{align}
  We will show that
  \[
    \phi(d_\mathbf{a}) = 0 \ \text{ for all } r > r_1,\ \mathbf{a} \in (\mathbb{Z}_{\geq 0} \times \mathbb{Z}_{> 0})^r.
  \]
  Let $r > r_1$ and $s \geq 0$ be such that $\phi(d_\mathbf{a}) = 0$ for all $\mathbf{a} = ((a_1,b_1),\dots,(a_r,b_r))$ with $\sharp\{i \mid \text{$b_i$ is even }\} \leq s$.
  Then, for each $(a,b) \in \mathbb{Z}_{\geq 0} \times \mathbb{Z}_{> 0,\mathrm{even}}$, we have
  \[
    \phi(d_{(a,b)*\mathbf{a}}) = \phi(d_{(a,b)}) \phi(d_\mathbf{a}) - \sum_{l=1}^r \phi(d_{((a_1,b_1),\dots,(a_l+a,b_l+b),\dots,(a_r,b_r))}) = 0.
  \]
  Hence, $\phi(d_\mathbf{b}) = 0$ for all $\mathbf{b} = ((a_1,b_1),\dots,(a_{r+1},b_{r+1}))$ with $\sharp\{i \mid \text{$b_i$ is even }\} \leq s+1$.
  Also, let $q \geq 0$ be such that $\phi(d_\mathbf{a}) = 0$ for all $\mathbf{a} = ((a_1,b_1),\dots,(a_r,b_r))$ with
  \[
    \sharp\{i \mid \text{$b_i$ is even }\} \leq s+1 \text{ and } \sharp\{ i \mid (a_i,b_i) = (0,1) \} = r-q > 0.
  \]
  We may assume that $(a_1,b_1),\dots,(a_q,b_q) \neq (0,1)$ and $(a_{q+1},b_{q+1}),\dots,(a_r,b_r) = (0,1)$.
  Then, for each $(a,b) \in \mathbb{Z}_{\geq 0} \times \mathbb{Z}_{> 0,\mathrm{odd}}$, we have
  \begin{align*}
    0 = \phi(d_{(a,b)*\mathbf{a}}) &= \phi(d_{(a,b)}) \phi(d_\mathbf{a}) - \sum_{l=1}^r \phi(d_{((a_1,b_1),\dots,(a_l+a,b_l+b),\dots,(a_r,b_r))})\\
    &= -(r-q)\phi(d_{((a_1,b_1),\dots,(a_q,b_q),(a,b+1),(0,1),\dots,(0,1))}).
  \end{align*}
  This implies that $\phi(d_\mathbf{a}) = 0$ for all $\mathbf{a} = ((a_1,b_1),\dots,(a_r,b_r))$ with
  \[
    \sharp\{i \mid \text{$b_i$ is even }\} \leq s+1 \text{ and } \sharp\{ i \mid (a_i,b_i) = (0,1) \} = r-q-1.
  \]
  By above, our claim follows by induction.

  For each $r \in \mathbb{Z}_{\geq 0}$ and $\mathbf{a} = ((a_1,b_1),\dots,(a_r,b_r)) \in (A \setminus \{(0,0)\})^r$, define $c_\mathbf{a}$ by
  \[
    c_\mathbf{a} := \phi(d_{((a_1,a_1+b_1),\dots,(a_r,a_r+b_r))}).
  \]
  By the definition of $d_{\mathbf{a}}$'s, equation \eqref{eq: phi(dbfa) = 0 AI2}, and Theorem \ref{thm: sufficient condition for evaluation2}, there exist $\delta_1,\dots,\delta_{r_1},\gamma_1,\dots,\gamma_{r_1} \in \mathbb{C}$ such that
  \[
    c_\mathbf{a} = m_\mathbf{a}((\delta_1,\gamma_1),\dots,(\delta_{r_1},\gamma_{r_1})) \ \text{ for all } \mathbf{a} \in (A \setminus \{(0,0)\})^*.
  \]
  As in the proof of Theorem \ref{thm: classification AI1}, we obtain
  \[
    \delta_i^2 = (\gamma_i^2+4)\gamma_i^2 \ \text{ for all } i = 1,\dots,r_1.
  \]

  For each $a \geq 1$, we have
  \[
    \phi(w_{+,0,2a}) = \phi(d_{(0,2a)}) = c_{(0,2a)} = \sum_{i=1}^r \gamma_i^{2a}.
  \]
  On the other hand, equation \eqref{eq: phi(w+a0) from AI1} implies that
  \[
    \phi(w_{+,0,2a}) = \sum_{i=1}^{r_1} ((\beta''_i)^2-4)^a \ \text{ for all } a \geq 1.
  \]
  By reordering the tuple $((\delta_1,\gamma_1),\dots,(\delta_{r_1},\gamma_{r_1}))$, we may assume that
  \[
    \gamma_i^2 = (\beta''_i)^2-4 \neq 0 \ \text{ for all } i = 1,\dots,r_1.
  \]

  Set $\beta_i := \delta_i \gamma_i^{-1}$.
  Then, we obtain
  \begin{align}\label{eq: beta2 = gamma2+4 neq 4}
    \beta_i^2 = \gamma_i^2+4 = (\beta''_i)^2 \neq 4.
  \end{align}

  As in the proof of Theorem \ref{thm: classification AI1}, for each $a \in \mathbb{Z}_{\geq 0}$, we have
  \begin{align*}
    \phi(w_{+,2a,0}) &= \sum_{i=1}^{r_1} \beta_i^{2a} + 4^a(\phi(w_{+,0,0})-r_1) \\
    &= \sum_{i=1}^{r_1} \beta_i^{2a} + 4^a(\nu_1+\nu_{-1}) \\
    \phi(w_{+,2a+1,0}) &= \sum_{i=1}^{r_1} \beta_i^{2a+1} + 4^a(\phi(w_{+,1,0})-\sum_{i=1}^{r_1} \beta_i).
  \end{align*}
  On the other hand, equation \eqref{eq: phi(w+a0) from AI1} implies that
  \begin{align*}
    \phi(w_{+,2a+1,0}) = 2 \cdot 4^a(\nu_1 - \nu_{-1}) + \sum_{i=1}^{r_1} (\beta''_i)^{2a+1} = 2 \cdot 4^a(\nu_1 - \nu_{-1}) + \sum_{i=1}^{r_1} \beta''_i \beta_i^{2a}.
  \end{align*}
  Comparing this and the previous identity, we obtain $\sum_{i=1}^{r_1} (\beta_i-\beta''_i) = 0$.
  Therefore,
  \[
    \phi(w_{+,2a+1,0}) = \sum_{i=1}^{r_1} \beta_i^{2a+1} + 2 \cdot 4^a(\nu_1-\nu_{-1}).
  \]
  
  Similarly, we have
  \begin{align*}
    &\phi(w_{-,2a,1}) = -\sum_{i=1}^{r_1} \beta_i^{2a}\gamma_i,\\
    &\phi(w_{-,2a+1,1}) = -\sum_{i=1}^{r_1} \beta_i^{2a+1}\gamma_i,
  \end{align*}
  for all $a \in \mathbb{Z}_{\geq 0}$.
  Summarizing, we have obtained
  \begin{align*}
    &\phi(w_{+,a,0}) = \nu_1 \cdot 2^a + \nu_{-1} \cdot (-2)^a + \sum_{i=1}^{r_1} \beta_i^a,\\
    &\phi(w_{-,a,1}) = -\sum_{i=1}^{r_1} \beta_i^a \gamma_i,
  \end{align*}
  for all $a \in \mathbb{Z}_{\geq 0}$.
  For each $i = 1,\dots,r_1$, there exists $\alpha_i \in \mathbb{C}^\times$ such that
  \[
    \beta_i = \alpha_i + \alpha_i^{-1} \text{ and } \gamma_i = -(\alpha_i-\alpha_i^{-1}).
  \]
  The condition \eqref{eq: beta2 = gamma2+4 neq 4} implies that $\alpha_i \neq \pm 1$.
  Thus, we complete the proof.
\end{proof}

\subsection{General case}
In this subsection, assume that the Dynkin diagram $I$ is irreducible.
When $\theta = \mathrm{id}$, let $\tilde{I}$ denote the Dynkin diagram of untwisted affine type corresponding to $I$, and $\tilde{S} := \emptyset$.
When $\theta \neq \mathrm{id}$, let $(\tilde{I},\tilde{S})$ denote the diagram in Figures \ref{fig 1}--\ref{fig 3} corresponding to $(I,\mu,S)$.
Let $(I,\mu,S)$ denote the corresponding triple.
Recall from Subsection \ref{subsect: involution} the elements $x_j,y_j,w_j \in \mathfrak{g}$.
For each $j \in \tilde{I}$ and $a \in \mathbb{Z}_{\geq 0}$, set
\[
  w_{j,a} := t_+^a \otimes w_j.
\]
Also, for each $j \in \tilde{I}$ such that $\mu(j) \neq j$ and $a \in \mathbb{Z}_{\geq 0}$, set
\[
  w_{j,-} := h_j - h_{\mu(j)}, \quad w_{j,-,a} := t_+^a t_-^a \otimes w_{j,-}.
\]
Then, $\mathcal{L}^{\mathrm{tw},0}$ has a basis
\[
  \{ w_{j,a} \mid j \in \tilde{I} \setminus \{ 0 \}, \ a \geq 0 \} \sqcup \{ w_{j,-,a} \mid j \in \tilde{I} \setminus \{ 0 \} \text{ such that } \mu(j) \neq j, \ a \geq 0 \}.
\]

From Figures \ref{fig 1}--\ref{fig 3}, we see that the center $\mathfrak{z}(\mathfrak{k})$ of $\mathfrak{k}$ is nontrivial if and only if $S = \{ 0,k \}$ for some $k \neq 0$.
In this case, $\mathfrak{z}(\mathfrak{k})$ is spanned by the fundamental coweight $\varpi_k^\vee$.

\begin{thm}\label{thm: classification general}
  For each $\phi \in (\mathcal{L}^{\mathrm{tw},0})^*$, the following are equivalent$:$
  \begin{enumerate}
    \item $\dim V(\phi) < \infty$.
    \item For each $j \in \tilde{I} \setminus \{ 0 \}$, the following hold:
    \begin{enumerate}
      \item if $a_{j,\mu(j)} = 0$, then there exist $n_j \in \mathbb{Z}_{\geq 0}$ and $\alpha_{j,1},\dots,\alpha_{j,n_j} \in \mathbb{C}^\times$ such that
      \begin{align*}
        &\phi(w_{j,a}) = \sum_{i=1}^{n_j} (\alpha_{j,i}+\alpha_{j,i}^{-1})^a, \\
        &\phi(w_{j,-,a}) = \sum_{i=1}^{n_j} (\alpha_{j,i}+\alpha_{j,i}^{-1})^a(\alpha_{j,i}-\alpha_{j,i}^{-1})
      \end{align*}
      for all $a \in \mathbb{Z}_{\geq 0}$,
      \item if $\mu(j) = j \notin S$, then there exist $n_j \in \mathbb{Z}_{\geq 0}$, $\alpha_{j,1},\dots,\alpha_{j,n_j} \in \mathbb{C}^\times$ such that
      \[
        \phi(w_{j,a}) = \sum_{i=1}^{n_j} (\alpha_{j,i} + \alpha_{j,i}^{-1})^a \ \text{ for all } a \in \mathbb{Z}_{\geq 0},
      \]
      \item if $\mu(j) = j \in S$, then there exist $\nu_{j,\pm 1} \in \mathbb{C}$, $n_j \in \mathbb{Z}_{\geq 0}$, $\alpha_{j,1},\dots,\alpha_{j,n_j} \in \mathbb{C}^\times \setminus \{ \pm 1 \}$ such that
      \[
        \phi(w_{j,a}) = \nu_{j,1} \cdot 2^a + \nu_{j,-1} \cdot (-2)^a + \sum_{i=1}^{n_j} (\alpha_{j,i} + \alpha_{j,i}^{-1})^a \ \text{ for all } a \in \mathbb{Z}_{\geq 0},
      \]
      and
      \[
        a_j^\vee \nu_{j,\pm 1} \in \mathbb{Z} \  \text{ if } 0 \notin \tilde{S}
      \]
      \item if $a_{j,\mu(j)} = -1$, then there exist $\nu_{j,\pm 1} \in \frac{1}{2}\mathbb{Z}_{\geq 0}$, $n_j \in \mathbb{Z}_{\geq 0}$, $\alpha_{j,1},\dots,\alpha_{j,n_j} \in \mathbb{C}^\times \setminus \{ \pm 1 \}$ such that
      \begin{align*}
        &\phi(w_{j,a}) = 2(\nu_{j,1} \cdot 2^a + \nu_{j,-1} \cdot (-2)^a + \sum_{i=1}^{n_j} (\alpha_{j,i} + \alpha_{j,i}^{-1})^a), \\
        &\phi(w_{j,-,a}) = 2\sum_{i=1}^{n_j} (\alpha_{j,i} + \alpha_{j,i}^{-1})^a(\alpha_{j,i} - \alpha_{j,i}^{-1})
      \end{align*}
      for all $a \in \mathbb{Z}_{\geq 0}$.
    \end{enumerate}
  \end{enumerate}
\end{thm}
\begin{proof}
  First, assume the condition (1).
  Then, Theorems \ref{thm: classification DeltaA1}, \ref{thm: classification A1}, \ref{thm: classification AI1}, and \ref{thm: classification AI2}, imply (2) except the condition on $a_j^\vee \nu_{j,\pm 1}$ in the case (c) (see also Remark \ref{rem: reform DeltaA1}).
  Hence, assume that there exists $j \in \{ 1,\dots,l \}$ such that $\mu(j) = j \in S$ and $0 \notin \tilde{S}$.
  From Figures \ref{fig 1}--\ref{fig 3}, we see that $\tilde{S} = \{j\}$ and $\tilde{I}$ is not of type $A_{2l}^{(2)}$.
  By equation \eqref{eq: sl2 triple 0}, we see that $(y_0,x_0,-w_0)$ forms an $\mathfrak{sl}_2$-triple.
  Also, by the definition of $y_0$, we have $y_0 v_\phi = 0$.
  Hence, by Theorem \ref{thm: classification A1}, there exist $n_0 \in \mathbb{Z}_{\geq 0}$ and $\alpha_{0,1},\dots,\alpha_{0,n_0} \in \mathbb{C}^\times$ such that
  \[
    \phi(-w_{0,a}) = \sum_{i=1}^{n_0} (\alpha_{0,i} + \alpha_{0,i}^{-1})^a \ \text{ for all } a \in \mathbb{Z}_{\geq 0}.
  \]
  This, together with equation \eqref{eq: w0}, implies that
  \[
    \sum_{i=1}^{n_0} (\alpha_{0,i}+\alpha_{0,i}^{-1})^a = a_j^\vee\nu_{j,1} \cdot 2^a + a_j^\vee\nu_{j,-1} \cdot (-2)^a + \sum_{k \in \tilde{I} \setminus \{0\}} a_k^\vee\sum_{i=1}^{n_k} (\alpha_{k,i}+\alpha_{k,i}^{-1})^a \ \text{ for all } a \in \mathbb{Z}_{\geq 0}.
  \]
  By Proposition \ref{prop: coincidence of p'}, we obtain
  \[
    \sharp\{ i \mid \alpha_{0,i} = \pm 1 \} = a_j^\vee \nu_{j,\pm 1} + \sum_{k \neq 0,j} a_k^\vee \cdot \sharp\{ i \mid \alpha_{k,i} = \pm 1 \}.
  \]
  This implies that $a_j^\vee \nu_{j,\pm 1} \in \mathbb{Z}$, as desired.

  Next, assume the condition (2).
  We only need to construct a finite-dimensional module which contains a highest weight module of highest weight $\phi$.
  For each $i \in I$, let $\varpi_i \in P^+$ denote the corresponding fundamental weight.
  Set
  \[
    V := \bigotimes_{j = 1}^l \bigotimes_{i=1}^{n_j} V(\varpi_j)_{\alpha_{j,i}}, \quad v := \bigotimes_j \bigotimes_i v_{\varpi_j} \in V.
  \]
  If there exist no $j \in \{ 1,\dots, l \}$ such that either $\mu(j) = j \in S$ or $a_{j,\mu(j)} = -1$, then we see that the submodule of $V$ generated by $v$ is a highest weight module of highest weight $\phi$.

  Suppose that there exists $j \in \{ 1,\dots, l \}$ such that $\mu(j) = j \in S$.
  If $|\tilde{S}| = 2$, then we have $\tilde{S} = \{ 0,j \}$ and the center of $\mathfrak{k}$ is spanned by $z := \varpi_j^\vee$.
  Hence, each $\nu \in \mathbb{C}$ can be regarded as an element of $P_\mathfrak{k}^+$ by
  \[
    \langle h_k, \nu \rangle = 0 \ \text{ for all } k \in \tilde{I} \setminus \tilde{S} \text{ and } \langle z, \nu \rangle = \nu.
  \]
  Then, considering
  \[
    V_\mathfrak{k}(\nu_{k,1})_1 \otimes V_{\mathfrak{k}}(\nu_{k,-1}) \otimes V
  \]
  instead of $V$, we see that the condition (1) follows.

  Next, assume that $|\tilde{S}| = 1$.
  In this case, $\mathfrak{k}$ is a semisimple Lie algebra whose Dynkin diagram is $\tilde{I} \setminus \{ j \}$.
  In particular, we can consider the fundamental weight $\varpi_0 \in P_\mathfrak{k}^+$ corresponding to $0$.
  Then, we have
  \[
    \langle w_j, \varpi_0 \rangle = \langle -\frac{1}{a_j^\vee}(w_0 - \sum_{k \neq 0,j} a_k^\vee w_k), \varpi_0 \rangle = -\frac{1}{a_j^\vee}.
  \]
  Since $a_j^\vee \nu_{j,\pm 1} \in \mathbb{Z}$, we can take $m_{\pm 1}, m'_{\pm 1} \in \mathbb{Z}_{\geq 0}$ such that
  \[
    m_{\pm 1} - \frac{m'_{\pm 1}}{a_j^\vee} = \nu_{j,\pm 1}.
  \]
  Hence, we only need to replace $V$ by
  \[
    V(\varpi_j)_1^{\otimes m_1} \otimes V_\mathfrak{k}(\varpi_0)_1^{\otimes m'_1} \otimes V(\varpi_j)_{-1}^{\otimes m_{-1}} \otimes V_\mathfrak{k}(\varpi_0)_{-1}^{\otimes m'_{-1}} \otimes V.
  \]

  Finally, suppose that there exist $j \in \{ 1,\dots,l \}$ such that $a_{j,\mu(j)} = -1$.
  In this case, $\tilde{I}$ is of type $A_{2l}^{(2)}$, $j = l$, and $w_l = 2(h_l + h_{l+1})$.
  Therefore, we need to consider
  \[
    V_\mathfrak{k}(\varpi_l)_1^{2\nu_{l,1}} \otimes V_\mathfrak{k}(\varpi_l)_{-1}^{2\nu_{l,-1}} \otimes V
  \]
  instead of $V$.
  Thus, we complete the proof.
\end{proof}

\input{table.tex}


\end{document}